\documentclass[11 pt]{amsart}
  
\usepackage{amsmath,amssymb,amsthm, mathrsfs, mathtools}
\usepackage{amscd,mathtools}
  \usepackage{geometry}
  \usepackage{graphicx,epstopdf}
  \usepackage{color,xcolor}
  \usepackage{enumerate}
  \usepackage{xspace}
  \usepackage{tipa}
  \usepackage{epsfig}
  \usepackage{pinlabel}
  \usepackage[all]{xy}

\usepackage{tikz}
\usepackage{caption, subcaption}
\usepackage{tikz-cd}
\usetikzlibrary{calc,decorations.pathreplacing}
\usepackage{tikz}
\usetikzlibrary{calc}
\usetikzlibrary{arrows}
\usetikzlibrary{decorations.pathreplacing}
\usetikzlibrary{intersections}

\usepackage{mathrsfs}
\usetikzlibrary{matrix}
\usetikzlibrary{positioning}
\DeclareMathAlphabet{\pazocal}{OMS}{zplm}{m}{n}
\tikzset{>=stealth}

  \usepackage{hyperref}

  \newcommand{\calN}{\mathcal{N}}

  \newcommand{\EE}{\mathbb{E}}

  \newcommand{\RR}{\mathbb{R}}

  \newcommand{\ZZ}{\mathbb{Z}}
   \newcommand{\thick}{\ensuremath{\operatorname{thick}}\xspace} 

  \newcommand{\gothic}{\mathfrak}

  \newcommand{\go}{{\gothic o}}

  \newtheorem{theorem}{Theorem}[section]
  \newtheorem{proposition}[theorem]{Proposition}
  \newtheorem{corollary}[theorem]{Corollary}
  \newtheorem{lemma}[theorem]{Lemma}

  \theoremstyle{definition}
  \newtheorem{definition}[theorem]{Definition}
  \newtheorem{claim}[theorem]{Claim}
  
  \newtheorem*{claim*}{Claim}

  \newtheorem*{question*}{Question}
  \newtheorem*{answer*}{Answer}
  \newtheorem*{application*}{Application}

  \theoremstyle{remark}
  \newtheorem{remark}[theorem]{Remark}
  \newtheorem*{remark*}{Remark}
  
  \newcommand{\lemref}[1]{Lemma~\ref{#1}}
  \newcommand{\propref}[1]{Proposition~\ref{#1}}

  \newcommand{\figref}[1]{Figure~\ref{#1}}

  \newcommand{\eqnref}[1]{Equation~\eqref{#1}}
  
  \newcommand{\fo}{{\mathfrak{o}}}
  
  \newcommand{\bfb}{{\textbf{b}}}

  \DeclareMathOperator{\Isom}{Isom}

  \newcommand{\nc}{\newcommand}

\nc{\dmo}{\DeclareMathOperator}

\nc{\Q}{\mathbb{Q}}
\nc{\R}{\mathbb{R}}
\nc{\Z}{\mathbb{Z}}
\nc{\C}{\mathbb{C}}
\nc{\cS}{\mathcal{S}}
\nc{\iso}{\cong}
\dmo{\Diff}{Diff}
\dmo{\Homeo}{Homeo}
\dmo{\dist}{dist}
\dmo\BDiff{BDiff}
\dmo\slide{sl}
\dmo\im{im}
\dmo\id{id}
\dmo\Fix{Fix}

  \newcommand{\pka}{\partial_{\kappa}}

  \newcommand{\MCG}{\ensuremath{\mathcal{MCG}}\xspace}  
   
  \newcommand{\MF}{\ensuremath{\mathcal{MF}}\xspace} 
  \newcommand{\PMF}{\ensuremath{\mathcal{PMF}}\xspace} 
   
  \newcommand{\Teich}{{Teichm\"uller }} 
  \newcommand{\Ham}{{Hamenst\"adt }}

  \newcommand{\sC}{{\sf C}}

  \newcommand{\sK}{{\sf K}}

  \newcommand{\sQ}{{\sf Q}}   
  \newcommand{\sR}{{\sf R}}

  \newcommand{\cc}{{\sf c}}

  \newcommand{\kk}{{\sf k}}   
       
  \newcommand{\mm}{{\sf m}}   
  \newcommand{\nn}{{\sf n}}

  \newcommand{\qq}{{\sf q}}   
  \newcommand{\rr}{{\sf r}}

\DeclareMathOperator{\diam}{diam}
\DeclareMathOperator{\Tei}{Teich}

  \newcommand{\param}{{\mathchoice{\mkern1mu\mbox{\raise2.2pt\hbox{$
  \centerdot$}}
  \mkern1mu}{\mkern1mu\mbox{\raise2.2pt\hbox{$\centerdot$}}\mkern1mu}{
  \mkern1.5mu\centerdot\mkern1.5mu}{\mkern1.5mu\centerdot\mkern1.5mu}}}

\DeclarePairedDelimiterX{\norm}[1]{\lvert}{\rvert}{#1}
\DeclarePairedDelimiterX{\Norm}[1]{\lVert}{\rVert}{#1}

  \renewcommand{\setminus}{{\smallsetminus}}
  
  \newcommand{\st}{\mathbin{\mid}} 
  \newcommand{\ST}{\mathbin{\Big|}} 
  \newcommand{\from}{\colon\thinspace}

\newcommand{\CAT}{\ensuremath{\operatorname{CAT}(0)}\xspace}

\begin{document}

  \title[Genericity of sublinearly Morse directions]{Genericity of sublinearly Morse directions in CAT(0) spaces and the Teichm\"uller space}
  
%

 \author   {Ilya Gekhtman}
 \address{Department of Mathematics, University of Toronto, Toronto, ON }
 \email{ilyagekh@gmail.com}

 \author   {Yulan Qing}
 \address{Shanghai Center for Mathematical Sciences, Fudan University, Shanghai}
 \email{yulan.qing@gmail.com}

  \author   {Kasra Rafi}
 \address{Department of Mathematics, University of Toronto, Toronto, ON }
 \email{rafi@math.toronto.edu}

 

\begin{abstract}
We show that the sublinearly Morse directions in the visual boundary of a rank-1 \CAT space with a geometric group action are generic in several commonly studied senses of the word, namely with respect to Patterson-Sullivan measures and stationary measures for random walks. We deduce that the sublinearly Morse boundary is a model of the Poisson boundary for finitely supported random walks on groups acting geometrically on rank-1 \CAT spaces. We prove an analogous result for mapping class group actions on \Teich space. Our main technical tool is a criterion, valid in any unique geodesic  metric space, that says that any geodesic ray with sufficiently many (in a statistical sense) strongly contracting segments is sublinearly contracting.
\end{abstract}
\maketitle

\section{Introduction}
A major theme in recent research in metric geometry has been to find evidence of abundance of hyperbolic behavior in non-hyperbolic spaces. Well studied examples of such spaces include \CAT spaces with rank-1 geodesics: geodesics which do not bound a flat of infinite diameter. In a sense, these can be considered as geodesics in \CAT spaces exhibiting hyperbolic behavior. To any \CAT space one can associate the visual boundary consisting of asymptotic equivalence classes of geodesic rays. Pairs of points on the visual boundary defining rank-1 geodesics form an open set invariant under the isometry group. Thus, they are assigned full weight by any full support measure on the square of the visual boundary boundary whose class is preserved and which is ergodic with respect a subgroup of isometries. However, the condition of being rank-1 for geodesic rays is not invariant under asymptotic equivalence, and thus cannot be used to define a isometry group invariant subset of the visual boundary. Moreover, the topological type of the visual boundary is not quasi-isometrically invariant \cite{CK00}. Thus, while genericity of rank-1 biinfinite geodesics encapsulates abundance of hyperbolic behavior in rank-1 \CAT spaces it is difficult to translate in terms of properties of groups acting on such spaces.

Qing and Rafi  \cite{subcontracting}, showed that a certain subset of the visual boundary consisting of limits of sublinearly Morse geodesics, when given a topology slightly different from the one induced from the visual boundary, called the sublinearly Morse boundary is in fact a quasi-isometry invariant. 

In this paper, we show that this subset of the visual boundary is "generic" in several reasonable senses of the word.
The visual boundary of a \CAT space $X$ carries several natural families of measures corresponding to limits of different averaging procedures over orbits of a group acting on $X$ properly and cocompactly. One such family is the so-called Patterson-Sullivan measure, studied in this context by Ricks \cite{Ricks}. These are the weak limits of ball averages in the metric on $X$ and are intimately related to the measure of maximal entropy on the unit tangent bundle of the geodesic flow on $X$.

\begin{theorem}\label{fullPSmeasuresublinear}
Let $G\curvearrowright X$ be a countable group acting properly discontinuously and by isometries  on a geodesically complete rank-1 \CAT space $X$. Assume the action is  temperate (the number of orbit points in a ball grows at most exponentially) and admits a finite Bowen-Margulis measure. (These assumptions are automatically satisfied when the action is cocompact). Let $\nu$ be the Patterson-Sullivan measure on the visual boundary of $X$. Then $\nu$ gives measure zero to the complement of sublinearly Morse directions.
\end{theorem}

The other family consists of stationary measures associated to random walks coming from finitely supported measures $\mu$ on $G$: these are weak limits of pushforwards in $X$ of convolutions of $\mu$.
\begin{theorem}\label{fullRWmeasuresublinear}
Let $G\curvearrowright X$ be a nonelementary action of a countable group of properly discontinuous and isometric actions on a rank-1 \CAT space $X$. Assume the action is temperate. Let $\nu$ be the stationary measure on the visual boundary associated to a finitely supported probability measure $\mu$ on $G$ whose finite support generates $G$ as a semigroup. Then $\nu$ gives measure zero 
to the complement of sublinearly Morse directions.
\end{theorem}


As a corollary to Theorem~\ref{fullRWmeasuresublinear} we obtain:
\begin{corollary}\label{Poissonboundary}
Let $G\curvearrowright X$ and $\nu$ be as in Theorem \ref{fullRWmeasuresublinear}. Then for a suitable sublinear function $\kappa$, the $\kappa$-Morse boundary of $X$ with either the subspace topology induced from the visual boundary or the Qing-Rafi topology is a topological model for the Poisson boundary of $(G,\mu)$.
\end{corollary}

We also prove analogous results for the action of mapping class group on \Teich space 
equipped with the \Teich metric.

\begin{theorem}\label{fullmeasuresublinearteich}
Let $S$ be a closed surface of genus at least $2$ and let $\Tei(S)$ be the \Teich space of $X$ with the \Teich metric. Let $\PMF$ be Thurston's boundary of \Teich space consisting of projective measured foliations.  Let $\nu$ be a measure on $\PMF$ which is either a normalized Thurston measure or the stationary measure associated to a finitely supported probability measure $\mu$ on the mapping class group $\MCG(S)$ such that the semigroup generated by the support of $\mu$ is a group containing at least two independent pseudo-Anosov elements. Then $\nu$ gives full measure to foliations 
associated to sublinearly Morse geodesics rays.
\end{theorem}

\begin{corollary}\label{Poissonboundaryteich}
Let $\mu$ be a probability measure on $\MCG(S)$ such that the semigroup generated by the support of $\mu$ is a group containing at least two independent pseudo-Anosov elements. Then, for a suitable sublinear function $\kappa$, the $\kappa$-Morse boundary of $\Tei(S)$  is a topological model for the Poisson boundary of $(\MCG(S),\mu)$.
\end{corollary}

We prove genericity of sublinearly Morse directions of the visual boundary by providing a criterion asserting that a geodesic ray containing enough strongly contracting subsegments is sublinearly contracting. 
More precisely, for any proper geodesic metric space $X$, we say a unit speed parametrized geodesic ray $\tau \from [0, \infty) \to X$ is 
\emph{$(N,C)$--frequently contracting} for constant $N,C>0$ if the following holds. 
For each $L>0$ and $\theta\in (0,1)$ there is an $R_0>0$ such that for $R>R_0$ and $t>0$ there is an interval 
of time $[s-L, s+L] \subset [t, t+ \theta R]$ and an $N$--contracting geodesic $\gamma$ 
(see Definition~\ref{Def:Contracting}) such that, 
\[
u \in [s-L , s+L]
\qquad\Longrightarrow \qquad
d(\tau(u), \gamma) \leq C. 
\]

\begin{theorem}\label{frequentlycontractingimpliessublinearlymorse} 
A frequently contracting geodesic in any proper geodesic metric space is sublinearly Morse.
\end{theorem}
We then use ergodic theoretic methods to prove genericity of frequently contracting geodesics in the context of rank-1 \CAT spaces and the \Teich space. 
For the Patterson-Sullivan (or normalized Thurston) measure, genericity of frequently contracting directions  is a consequence of the Birkhoff Ergodic Theorem and the ergodicity of the geodesic flow in rank-1 \CAT spaces and the \Teich space.
For stationary measures coming from random walks, the genericity of frequently contracting geodesics is derived from the double ergodicity of the Poisson boundary and is facilitated by a quantitative version of an argument (in the \Teich setting) of Kaimanovich-Masur \cite{Kaimanovich-Masur}, which was first used by Baik-Gekhtman-\Ham \cite{BGH}.

\subsection*{Related results} Kaimanovich \cite{KaiHyp} proved that the Poisson boundary of hyperbolic groups are realized on their Gromov boundary. For \CAT groups, Karlson-Margulis \cite{KarlssonMargulis} showed that random walk tracks geodesic rays sublinearly and  thus the visual boundary realizes the Poisson boundary of \CAT spaces on which a \CAT group acts geometrically. However, visual boundaries are in general not QI-invariant and therefore not group-invariant, as shown by Croke-Kleiner \cite{CK00}. Qing-Rafi-Tiozzo \cite{subcontracting} constructed $\kappa$--Morse boundaries for \CAT spaces that are QI-invariant and in the case of right-angled Artin groups, do realizes their Poisson boundaries. For mapping class groups, Kaimanovich-Masur showed that uniquely ergodic projective measured foliations with the corresponding harmonic measure can be identified with the Poisson boundary of random walks; Qing-Rafi-Tiozzo \cite{QRT2} showed that, when $\kappa=\log t$, the $kappa$-boundary of the Cayley graph of the mapping class group can be identified with the Poisson boundary of the associated random walks.

We expect Theorems \ref{fullPSmeasuresublinear} and \ref{fullRWmeasuresublinear} and Corollary \ref{Poissonboundary} to hold in a more general context of actions admitting a strongly contracting element.

In fact, versions of our Theorems~\ref{fullRWmeasuresublinear} and Theorem~\ref{fullmeasuresublinearteich} concerning stationary measures were recently obtained in this more general setting by Inhyeok Choi \cite{Choi}, who in place of our ergodic theoretic and boundary techniques uses a pivoting technique developed by Gouezel. Meanwhile, a Patterson-Sullivan theory on a certain quotient of the horofunction boundary for spaces admitting nonelementary group actions with contracting elements was recently obtained by Coulon \cite{Coulon} and Yang \cite{Yang2022}. However a critical ingredient of our Theorem~\ref{fullPSmeasuresublinear} and Corollary \ref{Poissonboundary}  involving Patterson-Sullivan (or Thurston) measures, namely the ergodicity of the square of the Patterson-Sullivan measure, is not known in that setting.



\section{Sublinearly Morse quasi-geodesic rays in proper metric space}\label{sublinear}
In geometric group theory, we are mainly interested in geometric properties of the associated spaces that are group-invariant. In the setting of finitely generated groups, group-invariance can be 
interpreted as \emph{quasi-isometries} between metric spaces and objects, which we introduce now.
\subsection{Quasi-isometries of groups and metric spaces}

\begin{definition}[Quasi-isometric embedding] \label{Def:Quasi-Isometry} 
Let $(X , d_X)$ and $(Y , d_Y)$ be metric spaces. For constants $\kk \geq 1$ and
$\sK \geq 0$, we say a map $\Phi \from X \to Y$ is a 
$(\kk, \sK)$--\textit{quasi-isometric embedding} if, for all $x_1, x_2 \in X$
$$
\frac{1}{\kk} d_X (x_1, x_2) - \sK  \leq d_Y \big(\Phi (x_1), \Phi (x_2)\big) 
   \leq \kk \, d_X (x_1, x_2) + \sK.
$$
If, in addition, every point in $Y$ lies in the $\sK$--neighbourhood of the image of 
$\Phi$, then $f$ is called a $(\kk, \sK)$--quasi-isometry. When such a map exists, $X$ 
and $Y$ are said to be \textit{quasi-isometric}. 

A quasi-isometric embedding $\Phi^{-1} \from Y \to X$ is called a \emph{quasi-inverse} of 
$\Phi$ if for every $x \in X$, $d_X(x, \Phi^{-1}\Phi(x))$ is uniformly bounded above. 
In fact, after replacing $\kk$ and $\sK$ with larger constants, we assume that 
$\Phi^{-1}$ is also a $(\kk, \sK)$--quasi-isometric embedding, 
\[
\forall x \in X \quad d_X\big(x, \Phi^{-1}\Phi(x)\big) \leq \sK \qquad\text{and}\qquad
\forall y \in Y \quad d_Y\big(y, \Phi\,\Phi^{-1}(x)\big) \leq \sK.
\]
\end{definition}


\subsection*{Geodesics and quasi-geodesic rays and segments} \label{Sec:Quadi-Geodesic}
Fix a base point $\go \in X$. A \emph{geodesic ray} in $X$ is an isometric embedding 
$\tau \from [0, \infty) \to X$ such that $\tau(0) = \go$. That is, by convention, a geodesic ray is always 
assumed to start from this fixed base-point. A \emph{quasi-geodesic ray} is a continuous quasi-isometric 
embedding $\beta \from [0, \infty) \to X$ such that $\beta(0) =\go$. That is, there are constants 
$\qq \geq 1$, $\sQ>0$ such that, for $s, t \in [0, \infty)$, we have 
\[
\frac{|s-t|}{\qq} - \sQ  \leq d_X \big(\beta (s), \beta(t)\big) 
   \leq \qq \, |s-t|+ \sQ.	
\]
The additional assumption that quasi-geodesics are continuous is not necessary,
but it is added for convenience and to make the exposition simpler. 
One can always adjust a quasi-isometric embedding slightly to make it continuous 
(see  \cite[Lemma III.1.11]{BH1}). 

Similar to above, a \emph{geodesic segment} is an isometric embedding 
$\tau \from [s,t] \to X$ and a \emph{quasi-geodesic segment} is a continuous 
quasi-isometric embedding $\beta \from [s,t] \to X$.

\subsubsection*{Notations}
We adopt the following notation for lines and segments in this paper. Suppose $\beta$ is a specified path, then
 \[[x, y]_{\beta} : \text{ the segment of } \beta \text{ from }x \in \beta \text{ to } y \in \beta. \]   
 In the special case where $\beta$ is a geodesic, we suppress the subscript, i.e. we use $[x, y]$ denote geodesic segments between the two points. If $\beta$ emanates from the base-point, then 
 \[ \beta|_{\rr}: \text{ the point on } \beta \text{  that is distance } \rr \text{ from } \go.\]

\subsection*{Contracting geodesics}
Let $Z$ be a closed subset of $X$ and $x$ be a point in $X$.  By $d(x, Z)$ we mean the set-distance 
between $x$ and $Z$, i.e. 
\[
d(x, Z) : = \inf \big \{ d(x, y) \st y \in Z \big \}. 
\]
Let
\[ \pi_{Z}(x) : = \big \{ y \st d(x, y) = d(x, Z) \big \} \]
be the set of nearest-point projections from $x$ to $Z$. Since $X$ is a proper metric space, 
$\pi_{Z}(x)$ is non empty. We refer to $\pi_{Z}(x) $ as the \emph{projection set} of $x$ to $Z$. 
For a quasi-geodesic $\beta$ and $x \in X$, we write $x_{\beta}$ to denote \emph{any} point in 
the projection set of $x$ to $\beta$. 

\begin{definition} \label{Def:Contracting}
We say a closed subset $Z \subset X$ is \emph{$N$--contracting} for a constant $N>0$ if,
for all pairs of points $x, y \in X$, we have
\[
d(x, y) < d(x, Z) \quad  \Longrightarrow  \quad d(x_{Z}, y_{Z}) \leq  N.
\]
Any such $N$ is called a \emph{contracting constant} for $Z$.
\end{definition}

\subsection{$\kappa$-Morse and $\kappa$-contracting sets.}
Now we introduce a large class of quasi-geodesic rays that are quasi-isometry invariant. Intuitively, these quasi-geodesics have a weak Morse property, i.e. their quasi-geodesics stay close 
asymptotically. To begin with, we fix a function that is sublinear in the following sense:
 
\subsubsection{Sublinear functions}
We fix a function 
\[
\kappa \from [0,\infty) \to [1,\infty)
\] 
that is monotone increasing, concave and sublinear, that is
\[
\lim_{t \to \infty} \frac{\kappa(t)} t = 0. 
\]
Note that using concavity, for any $a>1$, we have
\begin{equation} \label{Eq:Concave}
\kappa(a t) \leq a \left( \frac 1a \, \kappa (a t) + \left(1- \frac 1a\right) \kappa(0) \right) 
\leq a \, \kappa(t).
\end{equation}


\begin{remark}
The assumption that $\kappa$ is increasing and concave makes certain arguments
cleaner, otherwise they are not really needed. One can always replace any 
sublinear function $\kappa$, with another sublinear function $\overline \kappa$
so that $\kappa(t) \leq \overline \kappa(t) \leq \sC \, \kappa(t)$ for some constant $\sC$ 
and $\overline \kappa$ is monotone increasing and concave. For example, define 
\[
\overline \kappa(t) = \sup \Big\{ \lambda \kappa(u) + (1-\lambda) \kappa(v) \ST 
\ 0 \leq \lambda \leq 1, \ u,v>0, \ \text{and}\ \lambda u + (1-\lambda)v =t \Big\}.
\]
The requirement $\kappa(t) \geq 1$ is there to remove additive errors in the definition
of $\kappa$--contracting geodesics. 
\end{remark}

\begin{definition}[$\kappa$--neighborhood]  \label{Def:Neighborhood} 
For a closed set $Z$ and a constant $\nn$ define the $(\kappa, \nn)$--neighbourhood 
of $Z$ to be 
\[
\calN_\kappa(Z, \nn) = \Big\{ x \in X \ST 
  d_X(x, Z) \leq  \nn \cdot \kappa(x)  \Big\}.
\]

\end{definition} 

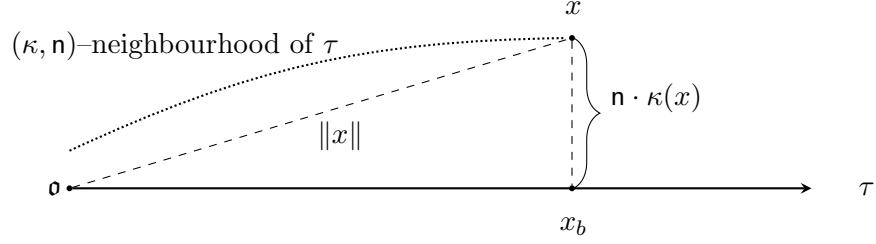
\begin{figure}[h!]
\begin{tikzpicture}
 \tikzstyle{vertex} =[circle,draw,fill=black,thick, inner sep=0pt,minimum size=.5 mm] 
[thick, 
    scale=1,
    vertex/.style={circle,draw,fill=black,thick,
                   inner sep=0pt,minimum size= .5 mm},
                  
      trans/.style={thick,->, shorten >=6pt,shorten <=6pt,>=stealth},
   ]

 \node[vertex] (a) at (0,0) {};
 \node at (-0.2,0) {$\go$};
 \node (b) at (10, 0) {};
 \node at (10.6, 0) {$\tau$};
 \node (c) at (6.7, 2) {};
 \node[vertex] (d) at (6.68,2) {};
 \node at (6.7, 2.4){$x$};
 \node[vertex] (e) at (6.68,0) {};
 \node at (6.7, -0.5){$x_{b}$};
 \draw [-,dashed](d)--(e);
 \draw [-,dashed](a)--(d);
 \draw [decorate,decoration={brace,amplitude=10pt},xshift=0pt,yshift=0pt]
  (6.7,2) -- (6.7,0)  node [black,midway,xshift=0pt,yshift=0pt] {};

 \node at (7.8, 1.2){$\nn \cdot \kappa(x)$};
 \node at (3.6, 0.7){$\Norm x$};
 \draw [thick, ->](a)--(b);
 \path[thick, densely dotted](0,0.5) edge [bend left=12] (c);
\node at (1.4, 1.9){$(\kappa, \nn)$--neighbourhood of $\tau$};
\end{tikzpicture}
\caption{A $\kappa$-neighbourhood of a geodesic ray $\tau$ with multiplicative constant $\nn$.}
\end{figure}

In this paper, $Z$ is either a geodesic or a quasi-geodesic. That is, we can write 
$\calN_{\kappa}(\tau, \nn)$ to mean the $(\kappa, \nn)$--neighborhood of the image of 
the geodesic ray $\tau$. Or, we can use phrases like 
``the quasi-geodesic $\beta$ is $\kappa$--contracting" or 
``the geodesic $\tau$ is in a $(\kappa, \nn)$--neighbourhood of the geodesic $c$". 

\begin{definition} \label{Def:Fellow-Travel}
Let $\beta$ and $\gamma$ be two quasi-geodesic rays in $X$. If $\beta$ is in some 
$\kappa$--neighbourhood of $\gamma$ and $\gamma$ is in some 
$\kappa$--neighbourhood of $\beta$, we say that $\beta$ and $\gamma$ 
\emph{$\kappa$--fellow travel} each other. This defines an equivalence
relation on the set of quasi-geodesic rays in $X$ (to obtain transitivity, one needs to change 
$\nn$ of the associated $(\kappa, \nn)$--neighbourhood). We refer to such an equivalence
class as a \emph{$\kappa$--equivalence class of quasi-geodesics}.
We denote the $\kappa$--equivalence class that contains $\beta$ by $[\beta]$ or we use 
the notation $\bfb$ for such an equivalence
class when no quasi-geodesic in the class is given. 
\end{definition}
A metric space is called a unique geodesic space if any two points can be connected by a unique geodesic.
\begin{lemma} \label{Lem:Unique} \cite{subcontracting}
Let $b \from [0,\infty) \to X$ be a geodesic ray in a unique geodesic space $X$. Then $b$ is the unique geodesic 
ray in any $(\kappa, \nn)$--neighbourhood of $b$ for any $\nn$. That is to say, distinct geodesic rays do not $\kappa$--fellow travel each other.
\end{lemma}


We recall the definition of $\kappa$--contracting and $\kappa$--Morse sets from \cite{QRT2}.

\begin{definition}[weakly $\kappa$--Morse] \label{Def:W-Morse} 
We say a closed subset $Z$ of $X$ is \emph{weakly $\kappa$--Morse} if there is a function
\[
\mm_Z \from \RR_+^2 \to \RR_+
\]
so that if $\beta \from [s,t] \to X$ is a $(\qq, \sQ)$--quasi-geodesic with end points 
on $Z$ then
\[
\beta[s,t]  \subset \calN_{\kappa} \big(Z,  \mm_Z(\qq, \sQ)\big). 
\]
We refer to $\mm_{Z}$ as the \emph{Morse gauge} for $Z$. We always assume
\begin{equation}
\mm_Z(\qq, \sQ) \geq \max(\qq, \sQ). 
\end{equation} 
\end{definition}

\begin{definition}[Strongly $\kappa$--Morse] \label{Def:S-Morse} 
We say a closed subset $Z$ of $X$ is \emph{strongly $\kappa$--Morse} if there is a function 
$\mm_Z\from \RR^2 \to \RR$ such that, for every constants $\rr>0$, $\nn>0$ and every
sublinear function $\kappa'$, there is an $\sR= \sR(Z, \rr, \nn, \kappa')>0$ where the 
following holds: Let $\eta \from [0, \infty) \to X$ be a $(\qq, \sQ)$--quasi-geodesic ray 
so that $\mm_Z(\qq, \sQ)$ is small compared to $\rr$, let $t_\rr$ be the first time 
$\Norm{\eta(t_\rr)} = \rr$ and let $t_\sR$ be the first time $\Norm{\eta(t_\sR)} = \sR$. Then
\[
d_X\big(\eta(t_\sR), Z\big) \leq \nn \cdot \kappa'(\sR)
\quad\Longrightarrow\quad
\eta[0, t_\rr] \subset \calN_{\kappa}\big(Z, \mm_Z(\qq, \sQ)\big). 
\]
\end{definition} 

\begin{definition}[$\kappa$--contracting] \label{Def:kappa-Contracting}
Recall that, for $x \in X$, we have $\Norm{x} = d_X(\go, x)$. 
For a closed subspace $Z$ of $X$, we say $Z$ is \emph{$\kappa$--contracting} if there 
is a constant $\cc_Z$ so that, for every $x,y \in X$
\[
d_X(x, y) \leq d_X( x, Z) \quad \Longrightarrow \quad
\diam_X \big( x_Z \cup y_Z \big) \leq \cc_Z \cdot \kappa(\Norm x).
\]
To simplify notation, we often drop $\Norm{\param}$. That is, for $x \in X$, we define
\[
\kappa(x) := \kappa(\Norm{x}). 
\]
\end{definition}

\begin{theorem}[\cite{QRT2}]
Let  $X$ be a proper geodesic space and let $\tau$ be a quasi-geodesic ray in $X$. 
Then 
\begin{enumerate}
\item If $\tau$ is $\kappa$-contracting then $\tau$ is both weakly and strongly 
$\kappa$-Morse.
\item $\tau$ is $\kappa$-weakly Morse if an only if $\tau$ is strongly $\kappa$-Morse.
\end{enumerate}
\end{theorem}

Lastly, a quasi-geodesic is called \emph{sublinearly Morse} if it is $\kappa$-Morse for some sublinearly growing function $\kappa$.
Two parametrized quasi-geodesics $\gamma_1,\gamma_2$ are said to be equivalent if their diverge sublinearly, i.e. \[d(\gamma_{1}(t),\gamma_{1}(t))/t\to 0.\]

\begin{definition}[$\kappa$-Morse boundary and sublinearly Morse boundary]
Given a sublinear function $\kappa$, let $\partial_{\kappa}X$ denote the set of equivalence classes of $\kappa$ Morse quasi-geodesics. Equipped with a coarse version of cone topology, we call this set the \emph{$\kappa$-Morse boundary} of $X$ and denote it $\pka X$ (for more details, see \cite{QRT2}). It is shown in \cite{QRT2} that $X\cup \partial_{\kappa}X$ with the coarse cone topology is a QI-invariant set and a metrizable topological space.

\end{definition}

 \section{Geodesics with enough contracting subsegments are sublinearly Morse}
\label{sublinearandrecurrent}
 In this section, we introduce the notion of frequently contracting geodesics which are geodesics that contain sufficiently many (in a statistical sense) strongly contracting subsegments.
We then give a criterion for a geodesic ray to be frequently contracting which can be conveniently checked using ergodic theory.
(\lemref{conditiontobefrequentlycontracting}). 

Finally, we prove that every frequently contracting geodesic is in fact
sublinearly Morse (Theorem~\ref{frequentlycontractingimpliessublinearlymorse} from the introduction, Theorem~\ref{main-frequentlycontractingimpliessublinearlymorse} in this section). 
In Sections 4-7, we use ergodic theory to show that in $\CAT$ spaces and \Teich space, frequently contracting geodesics are generic in several reasonable senses of the word.
\begin{definition}A unit speed parametrized geodesic ray $\tau \from [0, \infty) \to X$ is 
\emph{$(N,C)$--frequently contracting} for constant $N,C>0$ if the following holds. 
For each $L>0$ and $\theta\in (0,1)$ there is an $R_0>0$ such that for $R>R_0$ and $0 < t < (1-\theta)R$ there is an interval 
of time $[s-L, s+L] \subset [t, t+ \theta R]$ and an $N$--contracting geodesic $\gamma$ such that, 
\[
u \in [s-L , s+L]
\qquad\Longrightarrow \qquad
d(\tau(u), \gamma) \leq C. 
\]
That is, every subsegment of $\tau$ of length $\theta L$ contains a segment of length $R$ that is $C$--close to
an $N$--contracting geodesic $\gamma$. 
A bi-infinite geodesic $\tau$ is frequently contracting if the rays $t\to \tau(t)$ and $t\to \tau(-t)$ are 
both frequently contracting.
\end{definition}

\begin{definition}
If for some $t$, $\tau(t-L,t+L)$ is $C$--close to some $N$--contracting geodesic $\gamma$, 
we say (in analogy with Teichm\"uller space) $\tau(t)$ is in the thick part of $\tau$.  Define 
\begin{align*}
\thick_\tau(T) = \Big| & \big\{t\in [0,T]: \text{$\tau(t-L,t+L)$ is $C$--close  }\\
&\text{to some $N$--contracting geodesic $\gamma$}\big\} \Big|.
\end{align*}
That is, $\thick_\tau(T)$ is the amount of time $\tau[0,T]$ spends in the thick part. 
We now give a sufficient condition for a geodesic ray to be frequently contracting. 
\end{definition}
\begin{lemma} \label{conditiontobefrequentlycontracting} 
Let $\tau \from [0, \infty) \to X$ be a geodesic ray. Suppose there are constants $N, C>0$ such that for each 
$L>0$  there is a $m>0$ where 
\[
\lim_{T\to \infty}\frac {\thick_\tau(T)}T=m.
\] 
Then $\tau$ is frequently contracting.
\end{lemma}

\begin{proof} Suppose that $\tau$ is a geodesic ray satisfying the condition of the Lemma.
Then 
\begin{equation} \label{Eq:Double-Limit}
\lim_{s,t \to \infty} \frac{\thick_\tau(t) /t }{\thick_\tau(s) / s } = 1. 
\end{equation} 
Now assume, by way of contradiction, that $\tau$ is not $(N,C)$--frequently contracting. Then there are 
constants $0<\theta<1$ and $L>0$ and sequences $R_n\to \infty$ and $0\leq t_n\leq (1-\theta )R_n$ 
such that $[t_n,t_n+\theta R_n]_{\tau}$ contains no segment of length $2L$ that is $C$--close to some 
$N$--contracting geodesic $\gamma$. That is, 
 \[
 \thick_\tau(t_n) = \thick_\tau(t_n + \theta R_n). 
 \]
 Therefore, 
 \[
 \frac{\thick_\tau(t_n)/t_n }{\thick_\tau(t_n + \theta R_n) / (t_n + \theta R_n)} 
    = \frac{(t_n + \theta R_n)}{t_n} 
    \geq \frac{(t_n + \theta t_n)}{t_n} = (1 + \theta) >1. 
 \]
This contradicts \eqnref{Eq:Double-Limit}. The contradiction proves the desired result of this lemma. 
\end{proof}

Our goal is to show that, if $\tau$ is frequently contracting, then the diameter of the projection of disjoint 
balls to $\tau$ is sublinearly small.  It is in fact sufficient to show that the diameter
of the projection of a disjoint balls is smaller than every linear function. 
 
\begin{proposition} \label{Prop:Small-Linear-Contracting}
Let $\tau \from [0, \infty) \to X$ be an $(N,C)$--frequently contracting geodesic.  
Then for every $\theta>0$ there is $R_0>0$ such that for all $R \geq R_0$ the following holds. 
Assume 
\[
d(x,y) \leq d(x, \tau) 
\qquad\text{and}\qquad
d(\go, x) \leq R
\]
Then 
\[
d(\pi_\tau(x), \pi_\tau(y)) \leq \theta R. 
\]
 \end{proposition} 

We recall several well known facts regarding the properties of contracting geodesics. 

\begin{lemma} \label{Lem:BGIT} \cite[Corollary 3.4]{BF09}
There are constants $C_1, D_1>0$ depending on $N$ such that if $\gamma$ is 
$N$--contracting and the geodesic segment $[x,y]$ is outside of the $C_1$--neighborhood 
of $\gamma$ then the projection of $[x,y]$ to $\gamma$ has diameter at most $D_1$. 
\end{lemma} 

\begin{lemma} \label{Lem:Geodesic-Covers} \cite[Corollary 3.4]{BF09}
There is a constant $C_2>0$ depending only on $N$ such that, for a $N$--contracting geodesic
$\gamma$ and for $x, y \in X$, if $d(\pi_\gamma(x), \pi_\gamma(y)) \geq D_2$, then
the $C_2$--neighborhood of the geodesic segment $[x,y]$ contains the segment 
$[\pi_\gamma(x), \pi_\gamma(y)]_\gamma$. 
\end{lemma} 

\begin{proof}[Proof of \propref{Prop:Small-Linear-Contracting}]
Assume $\tau[s,t]$ is $C$ close to some $N$--contracting geodesic $\gamma$
with 
\[
d(\tau(s), \gamma(s')) \leq C 
\qquad\text{and}\qquad
d(\tau(t), \gamma(t')) \leq C 
\]
for some times $s < t$ and $s' < t'$ where $L = (t-s)$ is large. 

\subsection*{Claim} There is a $D_2$ (depending only on $N$ and specified in Lemma~\ref{Lem:Geodesic-Covers} )such that, for any $x \in X$, 
if $\pi_\tau(x) = \tau(u)$ for $u \leq s$ then $\pi_\gamma(x) = \gamma(u')$ for 
$u' \leq s' + D_2$.

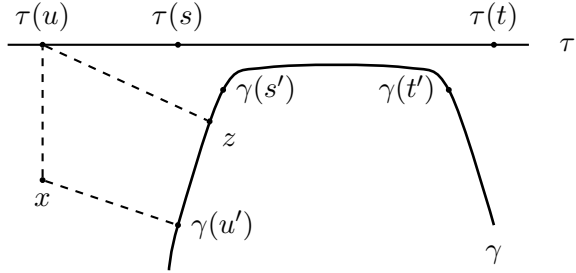
\begin{figure}[h!]
\begin{tikzpicture}[scale=0.6]
 \tikzstyle{vertex} =[circle,draw,fill=black,thick, inner sep=0pt,minimum size=.5 mm]
[thick, 
    scale=1,
    vertex/.style={circle,draw,fill=black,thick,
                   inner sep=0pt,minimum size= .5 mm},
                  
      trans/.style={thick,->, shorten >=6pt,shorten <=6pt,>=stealth},
   ]

  \node(o) at (0,0) {}; 
   \node[vertex](au) at (1,0) [label=above:$\tau(u)$] {}; 
     \node[vertex](as) at (4,0) [label=above:$\tau(s)$] {}; 
      \node[vertex](at) at (11,0) [label=above:$\tau(t)$] {}; 
    \node(a) at (12, 0)[label=right:$\tau$] {}; 
    
  \node[vertex](ys) at (5,-1) [label=right:$\gamma(s')$] {}; 
    \node[vertex](z) at (4.7,-1.7) [label=below right:$z$] {}; 
   \node[vertex](yt) at (10,-1) [label=left: $\gamma(t')$] {}; 
    \node[vertex](yu) at (4,-4) [label=right: $\gamma(u')$] {}; 
   
   \node[vertex](x) at (1,-3)[label=below:$x$] {};

     \draw [thick](o)--(a){};
  \draw  [thick,dashed](x)--(au){};
   \draw  [thick,dashed](x)--(yu){};
      \draw  [thick,dashed](au)--(z){};

          \node at (11,-4)[label=below:$\gamma$] {};

   \pgfsetlinewidth{1pt}
  \pgfsetplottension{.55}
  \pgfplothandlercurveto
  \pgfplotstreamstart
    \pgfplotstreampoint{\pgfpoint{3.8cm}{-5cm}}
  \pgfplotstreampoint{\pgfpoint{4cm}{-4cm}}
  \pgfplotstreampoint{\pgfpoint{5cm}{-1cm}}
  \pgfplotstreampoint{\pgfpoint{6cm}{-0.5cm}}
  \pgfplotstreampoint{\pgfpoint{9cm}{-0.5cm}}
  \pgfplotstreampoint{\pgfpoint{10cm}{-1cm}}
  \pgfplotstreampoint{\pgfpoint{11cm}{-4cm}}
  \pgfplotstreamend
  \pgfusepath{stroke} 

 \end{tikzpicture}
\caption{The point $z = \pi_\gamma (\pi_\tau (x))$ is near $\gamma(s')$. }
\label{Fig:z}
\end{figure}

This is because by \lemref{Lem:BGIT}, $z=\pi_\gamma (\pi_\tau(x))$
is near $\gamma(s')$ (see \figref{Fig:z}). If $\pi_\gamma(x) = \gamma(u')$ where $(u'-s')$ is larger than $D_2$, 
then by \lemref{Lem:Geodesic-Covers}, we have the a $C_2$--neighborhood of the geodesic
$[x, \pi_\tau(x)]$ contains the sub-segment $\gamma[s', u']$. 
Choose $w'$ such that $w'-s'$ is large and, 
\[ d(\gamma(w'), \gamma(w)) \leq C\]
for some $w$ where $(w-s)$ is large. Therefore, $\gamma(w')$ and hence $\tau(w)$ 
are much closer to $x$ than $\pi_\tau(x)$ which is a contradiction and thus the claim holds.

Now, if $d(\pi_\tau(x), \pi_\tau(y)) \geq \theta R$ for sufficiently large $R$, 
the segment 
$[\pi_\tau(x), \pi_\tau(y)]_\tau$ contains a subsegment $\tau[s, t]$
with $(t-s)\geq L$ that is $C$ close to some $\gamma$ (see \figref{Fig:L}). Then 
the projection of $x, y$ to $\gamma$ are $L-2D_2$ apart. Which means the 
$(D_2+C)$--neighborhood of the geodesic segment  $[x,y]$ covers the segment 
$\tau[s,t]$. Hence $d(x,y) > d(x, \tau)$, contradicting the assumption. This finishes the proof of \propref{Prop:Small-Linear-Contracting}. 
\end{proof}


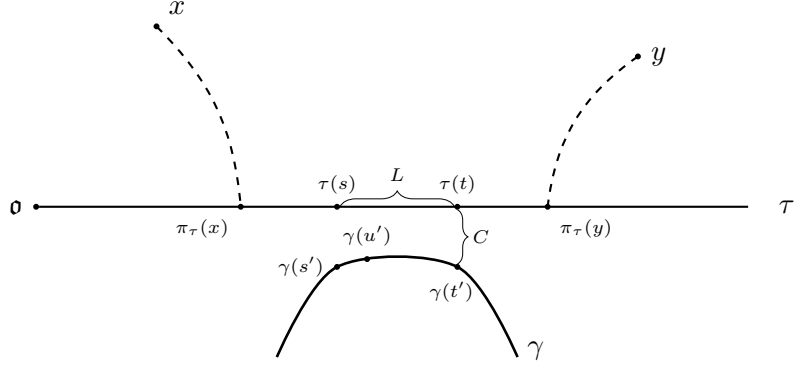
\begin{figure}[h!]
\begin{tikzpicture}[scale=0.8]
 \tikzstyle{vertex} =[circle,draw,fill=black,thick, inner sep=0pt,minimum size=.5 mm]
[thick, 
    scale=1,
    vertex/.style={circle,draw,fill=black,thick,
                   inner sep=0pt,minimum size= .5 mm},
                  
      trans/.style={thick,->, shorten >=6pt,shorten <=6pt,>=stealth},
   ]

  \node[vertex] (o) at (1,0)[label=left:$\go$] {}; 
    \node(a) at (13, 0)[label=right:$\tau$] {}; 
  \node[vertex](x) at (3,3) [label=above right:$x$] {}; 
   \node[vertex](y) at (11,2.5)[label=right:$y$] {}; 
  \node[vertex](x1) at (4.4, 0)[label=below left:\tiny$\pi_\tau(x)$] {}; 
    \node[vertex](s1) at (6, 0)[label=above:\tiny$\tau(s)$] {}; 
      \node[vertex](t1) at (8, 0)[label=above:\tiny$\tau(t)$] {}; 
       \node[vertex](s2) at (6, -1)[label=left:\tiny$\gamma(s')$] {}; 
       \node[vertex](s3) at (6.5, -0.87)[label=above:\tiny$\gamma(u')$] {}; 
      \node[vertex](t2) at (8, -1) {}; 
      
      \node at (7.9, -0.9)[label=below: \tiny $\gamma(t')$] {}; 
    \node[vertex](y1) at (9.5,0)[label=below right:\tiny$\pi_\tau(y)$] {}; 
    
	
	    \draw [very thin,
	decorate,
	decoration={brace,mirror,amplitude=6pt}] 
	(t1) -- (s1) node [midway,yshift=12pt] {\tiny $L$};
	
		    \draw [very thin,
	decorate,
	decoration={brace,mirror,amplitude=4pt}] 
	(t2) -- (t1) node [midway,xshift=9pt] {\tiny $C$};

          \node at (9.3,-1.9)[label=below:$\gamma$] {}; 
  
    \draw [thick](o)--(a){};
   \draw [thick, dashed]  (3,3)to [bend left=20] (4.4, 0);
     \draw [thick, dashed]  (11,2.5)to [bend right=24] (9.5, 0);
   
   \pgfsetlinewidth{1pt}
  \pgfsetplottension{.55}
  \pgfplothandlercurveto
  \pgfplotstreamstart
  \pgfplotstreampoint{\pgfpoint{5cm}{-2.5cm}}
  \pgfplotstreampoint{\pgfpoint{6cm}{-1cm}}
  \pgfplotstreampoint{\pgfpoint{8cm}{-1cm}}
  \pgfplotstreampoint{\pgfpoint{9cm}{-2.5cm}}
  \pgfplotstreamend
  \pgfusepath{stroke} 

 \end{tikzpicture}
\caption{If $d(\pi_\tau (x), \pi_\tau (y)) \geq \theta R$
the segment $[\pi_\tau(x), \pi_\tau(y)]_\tau$ contains a subsegment of
length $L$ that is $C$-close to some $\gamma$. }
\label{Fig:L}
\end{figure}


\begin{theorem} \label{main-frequentlycontractingimpliessublinearlymorse}
If $\tau$ is frequently contracting, then it is $\kappa$-contracting for some sublinear function
$\kappa$. Hence it is also $\kappa$-Morse. 
\end{theorem} 
 
 \begin{proof}
Assume for contradiction that $\tau$ is not $\kappa$--contracting for any sublinear function
$\kappa$. That is, there is a sequence of point $x^n, y^n \in X$ with $\Norm{x^n} \to \infty$, such that 
\[d_X(x^n, y^n) \leq d_X( x^n, \tau).\]
 However, we have
\[
\limsup_{n \to \infty} \frac{\diam_X \big( x^n_\tau \cup y^n_\tau \big)}{\Norm{x^n}} \geq 3 \theta >0. 
\]
Taking a subsequence, we can in fact assume that, for every $n$,  
\begin{equation} \label{Eq:2theta}
\frac{\diam_X \big( x^n_\tau \cup y^n_\tau \big)}{\Norm{x^n}} \geq 2 \theta.
\end{equation}
Let $R_0$ be the constant associated to $\theta$ given by Proposition~\ref{Prop:Small-Linear-Contracting} 
and let $n$ be such that $\Norm{x_n} \geq R_0$. Then for $R= \Norm{x_n}$,  
 Proposition~\ref{Prop:Small-Linear-Contracting}  implies that 
 \[
 \frac{\diam_X \big( x^n_\tau \cup y^n_\tau \big)}{\Norm{x^n}} \leq \theta,
 \]
 which contradicts \eqref{Eq:2theta}. The contradiction proves the corollary as desired. 
\end{proof}

\section{\CAT spaces, their boundaries and their isometries }\label{sectionCAT}
A proper geodesic metric space is \CAT if it satisfies a certain metric analogue of nonpositive 
curvature (see, for example, \cite{BH1}). Roughly speaking, a space is \CAT if geodesic triangles 
in $X$ are at least as thin as triangles in Euclidean space with the same side lengths. To be precise, for any 
given geodesic triangle $\triangle pqr$, consider the unique triangle 
$\triangle \overline p \overline q \overline r$ in the Euclidean plane with the same side 
lengths. For any pair of points $x, y$ on edges $[p,q]$ and $[p, r]$ of the 
triangle $\triangle pqr$, if we choose points $\overline x$ and $\overline y$  on 
edges $[\overline p, \overline q]$ and $[\overline p, \overline r]$ of 
the triangle $\triangle \overline p \overline q \overline r$ so that 
$d_X(p,x) = d_\EE(\overline p, \overline x)$ and 
$d_X(p,y) = d_\EE(\overline p, \overline y)$ then,
\[ 
d_{X} (x, y) \leq d_{\EE^{2}}(\overline x, \overline y).
\]

Here, we list some properties of proper \CAT spaces that are needed later. 

\begin{lemma}[\cite{BH1})] \label{Lem:CAT} 
 A proper \CAT space $X$ has the following properties:
\begin{enumerate}[i.]
\item It is uniquely geodesic, that is, for any two points $x, y$ in $X$, 
there exists exactly one geodesic connecting them. Furthermore, $X$ is contractible 
via geodesic retraction to a base point in the space. 
\item The nearest-point projection from a point $x$ to a geodesic line $b$ 
is a unique point denoted $x_b$. In fact, the closest-point projection map
\[
\pi_b \from X \to b
\]
is Lipschitz. 
\item Let $\beta \from [0 ,1] \to X$ is a continuous path and let 
$\gamma \from [0, 1] \to X$ be a geodesic segment such that $\beta(i) = \gamma(i),  i = 0, 1$. 
Then for every $0 \leq s \leq 1$, there exists $t$ such that $\pi_{\gamma}(\beta(t)) = \gamma(s)$.
\end{enumerate}
\end{lemma}

\subsection{The visual boundary of \CAT spaces.}
 \begin{definition}[visual boundary]
Let $X$ be a \CAT space. The \emph{visual boundary} of $X$, denoted $\partial_{\rm vis} X$ is the set of equivalence classes of unit speed geodesic rays (from any point in $X$) where two rays $\tau_1,\tau_2$ are equivalent if $\sup_{t>0}d(\tau_1(t),\tau_2 (t))<\infty$. 
\end{definition}

We describe the topology of the visual boundary by a neighbourhood basis:  
A neighborhood basis for a ray $\tau$ is given by sets of the form: 
\[U_{\rm vis}\big(\tau, r, \epsilon):=\{ \beta \in \partial_{\rm vis} X \, | \, d(\tau(t), \beta(t))<\epsilon \, \, \text{for all }\, t < r \}.\]

In other words, two geodesic rays are close if they have
geodesic representatives that start at the same point and stay close (are at most $\epsilon$ apart) for a long
time (at least $r$). By Proposition I. 8.2 in \cite{BH1}, for each ray $\tau$ and each $\go' \in X$, there is a unique geodesic ray $\tau'$ starting at $\go'$ with such $\tau$ and $\tau'$ are
a bounded distance from each other, that is, the visual boundary is independent of the base-point.

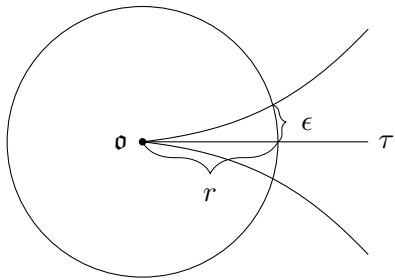
\begin{figure}[h]
\begin{center}
\begin{tikzpicture}[scale=0.6]

\node (x) [circle,fill,inner sep=1pt,label=180:$\go$] at (0,0) {};
\draw [name path=circle] (0,0) circle (3);

\draw [name path=line,thin] (0,0) to [bend right=20] (5,2.5);
\draw [thin] (0,0) to [bend left=20] (5,-2.5);
\draw (0,0) to (5,0) node [right] {$\tau$};

\draw [very thin,
	decorate,
	decoration={brace,mirror,amplitude=12pt}] 
	(0,0) -- (3,0) node [midway,yshift=-20pt] {$r$};

\draw [very thin,
	name intersections={of=line and circle},
	decorate,
	decoration={brace,amplitude=4pt}] 
	(intersection-1) -- (3,0) node [midway,xshift=12pt] {$\epsilon$};

\end{tikzpicture}
\end{center}
\caption{A neighbourhood basis for the visual topology.}
\label{}
\end{figure}

For bi-infinite geodesic $\tau \from (-\infty, \infty) \to X$ let $\tau_+$ be the geodesic ray
that is bounded distance from $\tau[0, \infty)$ and let $\tau_-$ be the geodesic ray
that is bounded distance from $\tau(\infty, 0]$.  We can think of $\tau_{+},\tau_{-}$ as points in 
$\partial_{\rm vis} X$ and we say $\tau$ \emph{connects} $\tau_-$ to $\tau_+$. We say $\tau_\pm$ are
limit points of $\tau$ in the visual boundary. Then $\tau$ is the unique geodesic with these limit 
points in $\partial_{\rm vis} X$. 

A flat strip of a space $X$ is a subset $F\subset X$ isometric to $\mathbb{R} \times I$ for some 
interval $I$ with the Euclidean metric.  It is a flat plane if $I=(-\infty,\infty)$ and a flat half 
plane if $I=[0,\infty)$. Call an infinite geodesic line in $X$ \emph{rank-1} if its image does 
not bound a flat half-plane in $X$ and \emph{zero width} if it does not bound any flat strip. 

Note, unless $X$ is Gromov hyperbolic, not any two points 
of $\partial_{\rm vis} X$ can be joined by a bi-infinite geodesic.
However, we have the following \cite[Lemma III.3.1]{Ballman}.

\begin{lemma}\label{strip} Let $X$ be a $\CAT$ space. Let $\tau: \mathbb{R} → X$ be a biinfinite geodesic which does not bound a flat strip of width $R > 0$. Then there are neighborhoods $U$ and $V$ in $\partial_{\rm vis} X$ of the endpoints of $\tau$ such that for any
$\alpha \in U$ and $\beta \in  V$ , there is a geodesic joining $\alpha$ to $\beta$. For any such geodesic $\tau'$, we have $d(\tau', \tau(0))<R$;
in particular, $\tau'$ does not bound a flat strip of width $2R$.
\end{lemma}

A \CAT space $X$ is said to be rank-1 if it has a rank-1 geodesic. The space $X$ 
is said to be geodesically complete if any geodesic segment can be extended to a bi-infinite geodesic.
Lemma \ref{strip} implies that if $X$ is rank-1 then for large enough $R>0$ pairs of points of $\partial_{\rm vis} X$ connected by a geodesic which does not bound a flat strip has nonempty interior. Moreover it can be easily seen to be $Isom(X)$ invariant.
\subsection{Group action on \CAT spaces}
For any element $g \in \Isom(X)$, the \emph{translation length} of $g$ is defined as 
\[
l(g)=\min \{d(x, gx) | x \in X \}.
\]
Any geodesic contained in $\{p\in X: d(p,gp)=l(p)\}$ is called an \emph{axis} of $g$. If a rank-1
geodesic $\gamma$ is an axis of an isometry $g \in \Isom(X)$, we call $\gamma$ a rank-1 axis and $g$ 
a rank-1 isometry. A rank isometry has exactly two fixed points in $\partial_{\rm vis} X$. 

Let $G$ be a group action on $X$ by isometries.  
The limit set $L(G)\subset \partial_{\rm vis} X$ of $G$ is the unique minimal closed $G$-invariant subset of $\partial_{\rm vis} X$.
An action by isometries $G\curvearrowright X$ is said to be rank-1 (and $G$ is said to be a rank-1 group) if it contains two rank-1 isometries whose pairs of fixed points in $\partial_{\rm vis} X$ are disjoint.
A geodesically complete \CAT space  admitting a rank-1 action $G \curvearrowright X$ always has a zero width geodesic with endpoints in $L(G)$(\cite{Ricks}, Proposition 8.14).
A cococompact action by isometries on a rank-1  \CAT space is always rank-1.

Let $SX$ denote the \emph{unit tangent bundle of $X$}, defined as the set of parametrized unit speed bi-infinite geodesics in $X$ endowed with the compact-open topology. 

Define \textit{footpoint projection} $ \pi_{\rm fp}:SX\to X$ by $ \pi_{\rm fp}(v)=v(0)$ 
for $v \in SX$. Let 
\[g_{t}:SX\to SX, \qquad g_t(\tau)(s)=\tau(s+t) \]
 be the geodesic flow.
Let $S_{G}X\subset SX$ denote the set of $v\in SX$ with $v_{\pm}\in L(G)$.
For $x,y\in X$ and $\zeta,\alpha  \in \partial_{\rm vis} X$ define the Busemann function 
\[
\beta_{\zeta}(x,y)=\lim_{z\to \zeta} d(x,z)-d(y,z)
\]
and the Gromov product
\[
\rho_{x}(\zeta,\alpha)=\lim_{z\to \alpha,w\to \zeta}(d(z,x)+d(w,x)-d(z,w))/2.
\]

Define a map $H_\go: SX\to \partial_{\rm vis} X\times \partial_{\rm vis} X\times \mathbb{R}$ by 
\[ H_{\go}(v)=(v_{+},v_{-},\beta_{\gamma}(\go,v(0)).\]
The restriction of $H_\go$ to the set $Z$ of zero width geodesics is one-to-one.
Let $[SX]$ be the set of equivalence classes of $SX$ under the equivalence relation $\sim$ given by 
$v\sim w$ if $H_{\go}(v)=H_{\fo}(w)$. This equivalence relation does not depend on the basepoint
$\go$.

\section{Genericity of frequently contracting geodesics with respect to Patterson-Sullivan measures}\label{Section5}
\label{recurrent}

The visual boundary $\partial_{\rm vis} X$  of a rank-1 \CAT space $X$ carries several natural classes of measures, corresponding to different averaging constructions over orbits of a properly discontinuous group action $G\curvearrowright X$.
In this section, we consider the Patterson-Sullivan measure, studied in this context by Ricks \cite{Ricks}. 

\begin{theorem}\label{psarereasonable}
Let $X$ be a proper geodesically complete \CAT space and let $G$ act on $X$ properly discontinuously and by isometries. Assume that $G\curvearrowright X$ admits a finite Bowen-Margulis measure (see below discussion; this assumption is satisfied when the action is cocompact). 
Let $\nu$ be the Patterson-Sullivan measure on $\partial_{\rm vis} X$. Then $\nu$ a.e.  point of $\partial_{\rm vis} X$ is frequently contracting.
\end{theorem}
We begin with some generalities. Let $G\curvearrowright X$ be any properly discontinuous action of a countable group on a metric space. For $\go \in X$ let $B_R(\go)$ be the ball of radius $R$ in $X$ centered at $\go$.

The (possibly zero or infinite) quantity 
\[\delta_X(G)=\lim \sup_{R\to \infty} R^{-1} \log |B_{R}(\go)\cap G \cdot \go|\] 
is called the critical exponent of the action $G\curvearrowright X$. 
It is easy to see that it does not depend on the basepoint $\go$.
If $G\curvearrowright X_i$ are properly discontinuous actions on  metric spaces and $f:X\to Y$ a map with $d(f(x),f(y))\leq Kd(x,y)+c$ it is easy to see that $\delta_Y(G)\leq K^{-1}\delta_X(G)$. 
If $X=Cay(G,S)$ is a Cayley graph of a finitely generated group then clearly $\delta_X(G)\leq \log |S|<\infty.$ Consequently, if $G\curvearrowright X$ is any properly discontinuous action of a finitely generated group on a metric space then $\delta_X(G)<\infty$ since the orbit map $Cay(G,S)\to X$ is coarsely Lipschitz. Moreover, $\delta_X(G)$ is always finite when $X$ admits some cocompact actions by a properly discontinuous group of isometries (\cite{Burger-Mozes}, Lemma 1.5 and Proposition 1.7).

On the other hand, when $G$ is non-amenable we always have $\delta_X(G)>0$ (see e.g. Giulio's email writeup, this is a thing everyone knows and no one writes). 
We now specialize to the case when $G\curvearrowright X$ is a propery discontinuous rank $1$ action on a $\CAT$ space such that there is some zero width geodesic with endpoints in $L(G)$, and assume $\delta_X(G)<\infty$. Since loxodromics with disjoint endpoints generate a free group in $Isom(X)$, $G$ is nonamenable so $\delta_X(G)>0$

The action $G\curvearrowright X$ is said to be \emph{divergent} if the Poincare series $\sum_{g\in G} e^{-sd(g\go,\go)}$ diverges at $s=\delta(G)$ and is said to be \emph{convergent} otherwise. If the action $G\curvearrowright X$ is properly discontinuous and cocompact it is necessarily divergent (\cite{Ricks}, Theorem 3).

A \emph{$\delta(G)$-conformal density} for $G\curvearrowright X$ is an absolutely continuous family of finite Borel measures $\nu_{x}, x\in X$ on $L(G)$  such that 
$$d\nu_x/d\nu_y (\zeta)=\exp(\delta(G)\beta_{\zeta}(y,x))$$ and $g\nu_x=\nu_{g^{-1}x}$ for any $x,y\in X$ and $g\in G$. Any such family is determined by any one of the measures $\nu_{\go}$, $\go \in X$ which we can normalize to be a probability measure.  A $\delta(G)$ conformal density always exists when $G$ is non-elementary and $\delta(G)<\infty$. When $G\curvearrowright X$ is divergent there is a unique conformal density for $G\curvearrowright X$ (see \cite{Link}, Theorems 10.1 and 10.2 and the remark after Theorem 10.1);  the measure $\nu_{\go}$ is called the \emph{Patterson-Sullivan measure}.
When $G\curvearrowright X$ is cocompact, the Patterson-Sullivan measure can be interpreted as the unique weak limit of ball averages over $G$ orbits in the metric on $X$.
More precisely, for $\go \in X$ we may consider the family of measures on $X$ given by $$\nu_{R,\go}=|G\cdot \go \cap B_{R}(\go)|^{-1}\sum_{g \go \in G\cdot \go \cap B_{R}(\go) }D_{g \go}$$ where $D_{g \go}$ denotes the point mass at $g \go$. Considering $\nu_{R,\go}$ as probability measures on the compact space $X\cup \partial_{\rm vis} X$, they converge (in the weak topology) as $R\to \infty$ to a scalar multiple of $\nu_{\go}$.

In the context of \CAT spaces,  conformal densities and Patterson-Sullivan measures were introduced by Ricks \cite{Ricks}.
Conformal densities can be used to construct a $G$ and geodesic flow invariant measure on $SX$ as follows.  The measure $\nu_\go \times \nu_\go $ gives full measure to endpoints of zero width geodesics, and thus after taking the product with the arc-length normalized Lebesgue measure $dL$ can be considered a measure on $SX$. Using the conformal density property, we can find a $G$-invariant and geodesic flow invariant Radon measure $\tilde{m}$ on $SX$ in the measure class of $\nu_\go \times \nu_\go \times dL$ (see \cite{Ricks} for details). This measure $\tilde{m}$ projects to a geodesic flow invariant measure $m$ on $SX/G$; both $m$ and $\tilde{m}$ are called the \emph{Bowen-Margulis measure}. When the Bowen-Margulis measure on $SX/G$ is finite (as is the case for instance when $G\curvearrowright X$ is cocompact) it is ergodic with respect to $g_t$ (\cite{Ricks}, Theorem 3) and the Patterson-Sullivan measure is the weak limit of ball averages as in the cocompact case \cite{Link2}. Recall that $S_{G}X\subset SX$ denotes the set of $v\in SX$ with $v_{\pm}\in L(G)$. We summarize the properties of the Patterson-Sullivan measure which we will use below.
\begin{lemma}\label{PSandBMproperties} \cite[Theorem 3]{Ricks}
Suppose $G\curvearrowright X$ is a non-elementary divergent action with $\delta(G)<\infty$ on a rank-1 \CAT space.  Assume there is some zero width geodesic with endpoints in the limit set $L(G)$. Assume the Bowen-Margulis measure $m$ on $ SX/G$ is finite. Then, it is ergodic with respect to the geodesic flow and gives full weight to zero width geodesics.  Furthermore, the Patterson-Sullivan measures $\nu_\go$ on $\partial_{\rm vis} X$ have full support on $L(G)$ and has no atoms, and the Bowen-Margulis measure $\tilde{m}$ has full support on $S_{G}X$.
\end{lemma}
Note, ergodicity of the geodesic flow $g_t$ on $(SX/G,m)$ by duality implies (in fact is equivalent to) ergodicity of the $G\curvearrowright (\partial_{\rm vis} X\times \partial_{\rm vis} X,\nu \times \nu)$.

We now prove Theorem \ref{psarereasonable}. We will show that for $\nu_\go$-a.e. $\zeta \in \partial_{\rm vis} X$ any geodesic ray converging to $\zeta$ satisfies the condition of Lemma  \ref{conditiontobefrequentlycontracting}. 
Since two geodesic rays converging to the same point of $\partial_{\rm vis} X$ are asymptotic, and the condition of Lemma  \ref{conditiontobefrequentlycontracting} is invariant under asymptotic equivalence classes, it suffices to prove that with respect to the Bowen-Margulis measure $\tilde{m}$ on $SX$, a.e.  geodesic satisfies the condition of Lemma \ref{conditiontobefrequentlycontracting}.

For any Borel set $V \subset SX/G$ the Birkhoff Ergodic Theorem and the ergodicity of $m$ with respect to the geodesic flow $g_t$ imply, for $m$-a.e.  $v\in SX/G$, we have
$$\lim_{T\to \infty}|\{t\in [0,T]:g_t v \in V\}|/T\to m(V).$$
Consequently, if $W$ is any $G$-invariant Borel subset of $SX$, then for $\tilde{m}$-a.e.  $v\in SX$, we have
\begin{equation}\label{E1}
    \lim_{T\to \infty}|\{t\in [0,T]:g_t v \in W\}|/T\to m(W/G).
\end{equation}
Note, $m(W/G)>0$ when $\tilde{m}(W)>0$.

For every $L>0, N>0, C>0$, define $W_{L,N,C}$ to be the set of $v\in SX$ such that $[ \pi_{\rm fp}(g_{-L}v),  \pi_{\rm fp}(g_{L}v)]$ is contained in a $C$-neighbourhood of an $N$-contracting geodesic. The set $W_{L,N,C}$ is $G$-invariant and Borel. Apply Equation~\ref{E1}, we get
$$\lim_{T\to \infty}|\{t\in [0,T]:g_t v \in W_{L,N, C}\}|/T\to m(W_{L,N, C}/G).$$ 
Any $v$ for which the left hand side of the above equation converges to a positive number  satisfies the condition of Lemma~\ref{conditiontobefrequentlycontracting}. Thus, it suffices to show that there is an $N, C>0$ such that for each $L>0$ we have $\tilde{m}(W_{L,N, C})>0$. 

To that end, let $N>0$ be large enough so that there exists an $N$-contracting axis, $v_0 \in SX$, for a hyperbolic isometry $g \in G$. Let $L, C>0$ and let $W_{v_0}\subset SX$ be the set of $v$ such that $d( \pi_{\rm fp}(g_{t}v),v_0)<C$ for all $t\in (-L,L)\}$. Clearly $W_{v_0}\subset W_{L,N, C}$.  Note $W_{v_0}$ is an open subset of $SX$.  Furthermore, it contains $v_0 \in SX$. Since $\tilde{m}$ has full support on $SX$ it follows that $\tilde{m}(W_{v_0})>0$ and thus $\tilde{m}(W_{L,N,C})>0$. 

\section{Stationary measures and random walks}\label{Section6}

The other family of measures on $\partial_{\rm vis} X$ we are interested in are stationary measures associated to random walks coming from finitely supported measures $\mu$ on $G$. 
In this section we prove: 

\begin{theorem}\label{harmonicarereasonable}
Let $X$ be a \CAT space and let $G\curvearrowright X$ be a rank-1 action.
Let $\nu$ be the stationary measure on  $\partial_{\rm vis} X$ coming from a finitely supported random walk on $G$. 
Then $\nu$-a.e. point of $\partial_{\rm vis} X$ is frequently contracting.
\end{theorem}

Let $G$ be an infinite group. Let $\mu$ be a symmetric probability measure on $G$ and
let $\mu^{\mathbb{Z}}$ be the product measure on $G^{\mathbb{Z}}$. Let $T \from G^{\mathbb{Z}}\to G^{\mathbb{Z}}$ be the following invertible transformation:
$T$ takes the two-sided sequence $(h_{n})_{n\in \mathbb{Z}}$
to the sequence $(\omega_{n})_{n\in \mathbb{Z}}$ with
$\omega_{0}=e$, 
$$\omega_{n}=h_{1}\cdots h_{n}\quad \text{ for }n>0$$ and
$$\omega_{n}=h^{-1}_{0}h^{-1}_{-1}\cdots h^{-1}_{-n+1}\quad \text{ for }n<0.$$

Similarly, let $\mu^{\mathbb{N}}$ be the product measure on $G^{\mathbb{N}}$.
Let $T_{+}:G^{\mathbb{N}}\to G^{\mathbb{N}}$ be the transformation that
takes the one-sided infinite sequence $(h_{n})_{n\in \mathbb{N}}$
to the sequence $(\omega_{n})_{n\in \mathbb{N}}$ with
$\omega_{0}=e$ and
$$\omega_{n}=h_{1}\cdots  h_{n}.$$

Let $\overline{P}$ be the pushforward measure
$T_{*}\mu^{\mathbb{Z}}$ on $G^{\mathbb{Z}}$and $P$ the pushforward measure
$T_{+*}\mu^{\mathbb{N}}$ on $G^{\mathbb{N}}$. The measure $P$ describes the distribution of $\mu$
sample paths, i.e. of products of independent $\mu$-distributed increments.
Let $\hat{\mu}$ be the measure on $G$ given by $\hat{\mu}(g)=\mu(g^{-1})$.
Let $\hat{P}$ be the pushforward measure
$T_{+*}\hat{\mu}^{\mathbb{N}}$ on $G^{\mathbb{N}}$. The measure space $(G^{\mathbb{Z}},\overline{P})$ is
naturally isomorphic to $(G^{\mathbb{N}},P)\otimes (G^{\mathbb{N}},\hat{P})$
via the map sending the bilateral path $\omega$ to the pair of unilateral
paths
$((\omega_{n})_{n\in \mathbb{N}},(\omega_{-n})_{n\in \mathbb{N}})$.

Let $\sigma: G^{\mathbb{Z}}\to G^{\mathbb{Z}}$ be the left Bernoulli shift:
$\sigma(\omega)_{n}=\omega_{n+1}.$
By basic symbolic dynamics, $\sigma$ is invertible,
measure-preserving and ergodic with respect to $\mu^{\Z}$.
Therefore, when restricted to sequences with $e$ at the $0$-th coordinate,
\[U=T\circ \sigma \circ T^{-1}\] is invertible, measure-preserving and
ergodic with respect to $\overline{P}$.
Note that for each $n\in \mathbb{Z}$,
$$(U\omega)_{n}=\omega^{-1}_{1}\omega_{n+1}$$
and more generally
$$(U^{k}\omega)_{n}=\omega^{-1}_{k}\omega_{n+k}.$$

Suppose $G$ acts continuously on an infinite Hausdorff space $B$. A Borel probability measure $\nu$ on $B$ is called \emph{$(G,\mu)$-stationary} if $$\nu(A)=\sum_{g\in G}\nu(g^{-1}A)\mu(g)$$ for all Borel $A\subset B$. The following is classical, see e.g. \cite{GGPY}, Theorem 9.4.
\begin{lemma}\label{stationarymeasureshavefullsupportandnoatoms}
Suppose $G\curvearrowright B$ is a minimal action on a compact Hausdorff space such that every $G$ orbit in $B$ is infinite. Let $\nu$ be a stationary measure on $B$. Then $\nu$ has no atoms and has full support on $B$.
\end{lemma}

Suppose now we have a bordification $Z=X\cup B$ of a metric space $X$, such that for any basepoint $\go \in X$ and $P$-a.e.  sample path $\omega=(\omega_n)_{n\in \mathbb{N}}$ the sequence $\omega_n \go$ converges to a point $\omega_\infty \in B$ independent of the basepoint $\go$. The probability measure on $B$ given by $\nu(A)=P(\omega:\omega_\infty \in A)$ is then clearly a stationary measure on $B$; moreover  for $P$-a.e. $\omega\in G^\mathbb{N}$ the pushforward measures $\omega^{*}_{n}\nu$ weakly converges to an atom concentrated at some $\omega_\infty \in B$. A space $B$ with a stationary measure $\nu$ satisfying the last condition is called a $(G,\mu)$ boundary. A  $(G,\mu)$ boundary $(B,\nu)$ is said to be a Poisson boundary of $(G,\mu)$ if it is maximal in the sense that for any other $(G,\mu)$ boundary  $(B',\nu')$  there is a $G$ equivariant measurable surjection $B\to B'$. The Poisson boundary is unique up to $G$-equivariant measurable isomorphism.

Karlsson and Margulis \cite{KarlssonMargulis} showed that under mild conditions the visual boundary of a \CAT space provides a model for the Poisson boundary of a group acting on the space. 
\begin{theorem}\cite{KarlssonMargulis}
Let $X$ be a \CAT space with basepoint $\go\in X$ and let $G\curvearrowright X$ be a nonamenable group acting on $X$ by isometries with bounded exponential growth. Let $\mu$ be a probability measure on $G$ whose finite support generates $G$ as a semigroup. Then for $P=P^\mu$-a.e.  $\omega \in G^\mathbb{N}$, $\omega_n \go$ converges to a point $\omega_\infty \in \partial_{\rm vis} X$. Moreover, $(\partial_{\rm vis} X, \nu)$ is a model for the Poisson boundary of $(G,\mu)$ where $\nu$ is the stationary measure on $\partial_{\rm vis} X$ defined by $\nu(A)=P(\omega_\infty \in A)$. 
\end{theorem}
Le Bars \cite[Theorem 1.1]{LeBars} proved that if in addition $G\curvearrowright X$ is a rank-1 action, the above measure $\nu$ is the unique $\mu$ stationary measure on $\partial_{\rm vis} X$.
We also recall the following theorem of Kaimanovich.

\begin{lemma}\label{Poissonergodic}\cite{KaiHyp}
The action of any group $G$ on the square of its Poisson boundary with respect to the square of the stationary measure associated to a symmetric random walk preserves the measure class and is ergodic.
\end{lemma}


Let $X$ be a rank-1 \CAT space. For $c>0$, define $A_c$ to be the set of pairs of points of $\partial_{\rm vis} X$ which are the endpoints of a rank-1 biinfinite geodesic that does not bound a flat strip of width $>c$. The set $A_c$ is an $\Isom(X)$-invariant subset of $\partial_{\rm vis} X$ and, for $c$ large enough, it has nonempty interior by Lemma \ref{strip}. Moreover, if $G\curvearrowright X$ is a rank-1 action, $A_c\cap L(G)$ is nonempty for large enough $c$. 
Consequently if $\nu$ is any nonatomic probability measure on $L(G)\subset \partial_{\rm vis} X$ with full support on $L(G)$ we have $\nu \times \nu (A_c)>0$ for large enough $c$. If in addition the $G$ action on $\partial_{\rm vis} X \times \partial_{\rm vis} X$ preserves and is ergodic with respect to the measure class of $\nu \times \nu$, the $\Gamma$ invariance of $A_c$ implies $\nu \times \nu(L(G)\setminus A_c)=0$.

We thus have:
\begin{lemma} \label{stationaryproperties}
Let $X$ be a geodesically complete rank-1 \CAT space and $G\curvearrowright X$ a rank-1 group action. Let $\mu$ be a symmetric probability measure on $G$ whose finite support generates $G$ and $\nu$ the associated stationary measure on $\partial_{\rm vis} X$. Then $\nu\times \nu$ is ergodic with respect to the $G$ action and for some $c>0$ gives full weight to pairs of points defining geodesics which do not bound a flat strip of width greater than $c$.
\end{lemma}

By Kingman's Ergodic Theorem and the nonamenability of $G$,  there is an $l=l(\mu)>0$ such that 
\begin{equation}\label{kingman}
l(\mu)=\lim_{n\to \infty}d(\omega_n \go,\go)/n 
\qquad \text{for} \qquad P\text{-a.e. }\omega.
\end{equation}
We refer to $l(\mu)$ as the \emph{drift} of the random walk $(G, \mu)$ with respect to the metric $d$. Karlsson-Margulis proved that for  $P$-a.e.  $\omega$ there is a parametrized unit-speed geodesic $\tau \in SX$ such that 
\begin{equation}\label{sublineartracking}
    d(\tau(ln),\omega_n \go)/n\to 0.
\end{equation} 
Recall that a point  $\zeta \in \partial_{\rm vis} X$ is a geodesic ray starting at $\go$. 

For the rest of the section, we assume that  $X,\nu,\mu$ are as in Lemma \ref{stationaryproperties}. Our goal is to show the following.
\begin{theorem}\label{aesublinear}
For any $\nu$-a.e.  $\zeta \in \partial_{\rm vis} X$, $\zeta$ is sublinearly Morse.
\end{theorem}
We will prove this by showing that $\nu$-a.e. $\zeta$ is frequently contracting.
In fact, we will prove a stronger statement:
\begin{proposition}\label{longfellowtravel}

Let $g_0\in G$ be a rank-1 element and let $\gamma^0$ be its axis. Then,  for $\nu$-a.e. $\zeta$, there is a $C>0$ such that the following holds. For any $b>a>0$ and $L>0$, there exists an $R_0>0$ such that, for any $R>R_0$, there is a $g \in G$ such that $\zeta([aR,bR])$ contains a subsegment of length $L$ that is $C$-close to $g \cdot \gamma^0$ .
\end{proposition}

 For each $\zeta_1,\zeta_2\in \partial_{\rm vis} X$,
 let $\Psi(\zeta_{1},\zeta_{2})$ be the set of unit speed, biinfinite geodesics with endpoints $\zeta_1,\zeta_2\in \partial_{\rm vis} X$. Let $\Psi(\zeta_{1},\zeta_{2},\go)$ be the set of unit speed parametrizations of such geodesics $\gamma$ such that $\gamma(0)$ is at minimal distance from $\go$.
For a bilateral sample path $\omega$ converging to $\omega_-,\omega_+ \in \partial_{\rm vis} X$ write $\Psi(\omega,\go)$ and $\Psi(\omega)$  instead of  
$\Psi(\omega_{-},\omega_{+},\go)$ and $\Psi(\omega_+,\omega_{-})$. 
Similarly for an unparametrized  biinfinite geodesic $\gamma$ we write $\gamma_x$ for the unit speed parametrization with $\gamma_x(0)$ at minimum distance from any $x \in X$.

Proposition \ref{longfellowtravel} will follow from Proposition~\ref{longfellowtravelbilateral}:

\begin{proposition}\label{longfellowtravelbilateral}

Let $g_0\in G$ be a rank-1 element and let $\gamma^0$ be its axis.
Then there is a $C>0$ such that for $\overline{P}$-a.e. biinfinite sample path $\omega$ any parametrization of any biinfinite 
geodesic $\gamma\in \Psi(\omega)$ satisfies the following. For any $\infty>b>a>-\infty$ and $L>0$ there is an $R_0=R_0(a,b,L)>0$ 
such that for any $R>R_0$ there is a $g\in G$ such that $\gamma([aR,bR])$ contains a subsegment of length $L$ that is $C$-close 
to $g \cdot \gamma^0$. 
\end{proposition}

Note that if the assumptions of Proposition \ref{longfellowtravelbilateral} holds for one unit speed parametrization of a geodesic, 
it is also satisfied by any such parametrization (for a different $R_0$). Also note that, unlike in Proposition \ref{longfellowtravelbilateral}, the $C>0$ in Proposition \ref{longfellowtravel} depends on the $\zeta \in \partial_{\rm vis} X$ but this is enough for our purposes.

\begin{proof}[Proof of Proposition \ref{longfellowtravel} assuming Proposition \ref{longfellowtravelbilateral}]
Proposition \ref{longfellowtravelbilateral} implies that there is $C>0$ such that, for $\nu$-a.e. $\zeta \in \partial_{\rm vis} X$, there is a unit speed parametrized biinfinite geodesic $\gamma_\zeta$ with the property that for any $\infty>b>a>0$ and $L>0$ there is an $R_0=R_0(a,b,L)>0$ such that for any $R>R_0$ there is a $g\in G$ such that $\gamma([aR,bR])$ contains a subsegment of length $L$ that is $C$-close to $g \cdot \gamma^0$. Let $\tau_\zeta$ be the ray from $\go$ converging to $\zeta$. Let $D=\sup d(\gamma_{\zeta}(t),\tau_{\zeta}(t))<\infty$. Then $\tau_\zeta$ satisfies the condition of Proposition \ref{longfellowtravelbilateral} with $C+D$ in place of $C$ and $R_0+D$ in place of $R_0$.
\end{proof}
The remainder of this section is devoted to the proof of Proposition~  \ref{longfellowtravelbilateral}. Let $\Omega(L,C,R)$ be the set of sample paths $\omega\in G^{\mathbb{Z}}$ such that, for all $\gamma \in \Psi(\omega)$, we have
$d(\go,\gamma)<R/10$ and there exist  $g\in G$ and  $t\in [-R/2+L,R/2-L]$
such that $\gamma_{\go}(t-L,t+L)$ is $C$-close to $g \cdot \gamma^0$.

\begin{lemma}\label{decaylemma}
There is a $C>0$ such that for all $L>0$ there is a positive function $f$ with 
\[
\lim_{R\to \infty}f(R)=0 
\qquad\text{and}\qquad  
\overline{P}(\Omega(L,C,R))>1-f(R).
\]
\end{lemma}
We first prove of Proposition \ref{longfellowtravelbilateral} assuming Lemma \ref{decaylemma} and will prove Lemma \ref{decaylemma} afterwards. 
\begin{proof}[Proof of Proposition \ref{longfellowtravelbilateral} assuming Lemma \ref{decaylemma}]
Assume, without loss of generality, that $a>0$ (this is possible since $[a,b]$ has a subsegment contained in either $(0,\infty)$ or $(-\infty,0)$ and in the latter case we can reverse the orientstion of the geodesic).   
Let $\Omega_{0}\subset 
G^\mathbb{Z}$ denote the set of 
all $\omega$ such that 
\[
d(\omega_{\pm i}\go,\go)/i\to l \qquad\text{and}\qquad 
\omega_{\pm i} \go \to \zeta_\pm \in \partial_{\rm vis} X
\]
with $(\zeta_-,\zeta_+)$ having width at most $c$. By Lemma~\ref{stationaryproperties} and Equation~\ref{kingman},  $\overline{P}(\Omega_0)=1$. Consider $\omega\in \Omega_0$.
Choose $R>0$ large enough so that $$1-f(R)>(b-a)/(10a+10b).$$
Assume $U^{i}\omega \in \Omega(L,C,R)$ for some $i \in \ZZ$. This is equivalent 
to saying, for all $\gamma \in \Psi(\omega)$, $d(\omega_{i}\go,\gamma)<R/10$
and there exists $t \in (-R/2+L,R/2-L)$ and $g \in G$ such that 
\[
\gamma_{\omega_{i}\go}(t-L,t+L) \ \text{is $C$--close to  $g \cdot \gamma^0$}.
\]
Therefore, we can choose time $t_i$ with $|t_{i}-d(\omega_{i}\go,\gamma_{\go}(0))|<R/10$
and $g_i\in G$ such that, for all $\gamma \in \Psi(\omega)$, 
\[
\gamma_{\go}(t_{i}-L,t_{i}+L)\ \text{ is $C$--close to $g_i \cdot \gamma^0$}. 
\]
Let $s_{i}(\gamma)=d(\omega_{i}\go,\gamma_{\go}(0))$. Let 
\[
d=\sup \big\{d(\go,g\go)\mid g\in {\rm supp}(\mu)\big\}.
\]
Note since $d(\omega_{i}\go, \omega_{i+1}\go)\leq d$ for all $i$, for every $t>d(\go,\gamma)$ 
there is some $i(t)$ with \[|t-s_{i(t)}(\gamma)|<d.\]
Hence, for large enough (depending on $\omega$) $n$, for all $\gamma\in \Psi(\omega)$, if there is an $i$ with $$U^{i}\omega \in \Omega(L,C,R)$$ and 
\begin{equation} \label{Eq:si}
(2a+b)n/3\leq s_{i}(\gamma)\leq (a+2b)n/3
\end{equation} 
then $\gamma_{\go}([an,bn])$ has a connected segment of length $L$ that is $C$--close to 
$g \cdot \gamma^0$, for some $g\in G$. Moreover, since $s_{i}(\gamma)/i \to l$, 
for large enough $n$, we have, \eqref{Eq:si} holds whenever 
\begin{equation} \label{Eq:i}
\frac{(3a+2b)n}{5l}\leq i\leq\frac{(2a+3b)n}{5l}.
\end{equation} 

Hence, unless $\gamma_{\go}([an,bn])$ has a connected segment of length $L$ that is 
$C$--close to $g \cdot \gamma^0$ for some $g\in G$, we have 
$$U^{i}\omega \notin \Omega(L,C,R)$$ 
for any $i$ as in \eqref{Eq:i}.
This implies that if $N=N(n)$ is the smallest integer less than $\frac{(2a+3b)n}{5l}$ we have 
$$\frac{|\{i\in [0,N-1]\mid U^{i}\omega \in \Omega(L,C,R)\}|}{N}\leq 1-\frac{b-a}{2(2a+3b)}.$$
Consequently, if for infinitely many $n$  $\gamma_{\go}([an,bn])$ has no connected segment of length $L$ that is $C$--close to $g \cdot \gamma^0$ for some $g\in G$, we have   
$$\liminf_{N\to \infty}\frac{|\{i\in [0,N-1]\mid U^{i}\omega \in \Omega(L,C,R)\}|}{N}\leq 1-\frac{b-a}{2(2a+3b)}.$$
On the other hand, by the Birkhoff ergodic theorem, for $\overline{P}$-a.e $\omega$, we have
\begin{align*}
\lim_{N\to \infty}\frac{|\{i\in [0,N-1]\mid U^{i}\omega \in \Omega(L,C,R)\}|}{N}   &=\overline{P}(\Omega(L,C,R))\\
&>1-\frac{(b-a)}{(6a+6b)} > 1-f(R) \qedhere 
\end{align*}
\end{proof}

Finally, we prove Lemma \ref{decaylemma}.
\begin{proof}[Proof of  Lemma \ref{decaylemma}]
Let $C$ be such that for $\overline{P}$-a.e. $\omega$ the points $\omega_\pm$ are connected by a geodesic of width at most $C$. This means that for $\overline{P}$ a.e. $\omega$,  $\Psi(\omega)$ is nonempty and there is an $R(\omega)\in \mathbb{N}$ such that for all $\gamma \in \Psi(\omega)$, we have
$d(\go,\gamma)<R(\omega).$  Thus, the $\overline{P}$ measure of $\omega\in G^{\mathbb{Z}}$ such that
$d(o,[\omega_{-},\omega_{+}])<R/10$ converges to 1 with $R$.

It now suffices to show that there is a $C>0$ such that for each $L>0$ the $\overline{P}$ measure of the set of $\omega$ such that for each $\gamma\in \Psi(\omega)$ the segment $\gamma_{\go}(-R/2,R/2)$ contains a length $L$ subsegment $C$ close to $g \cdot \gamma_0$ for some $g\in G$ converges to $1$ as $R\to \infty$.
Let $\Lambda(L,C)$ be the set of biinfinite sample paths $\omega$ such that $d(\gamma_{\go}(t),\gamma^0 (t))<C$ for all $\gamma\in \Psi(\omega)$ and $|t|\leq L$.

\begin{claim} 
There is a $C>0$ such that $\overline{P}(\Lambda(L,C))>0$ for all $L$.
\end{claim} 
\begin{proof}
Let $C$ be large enough so that the periodic rank-1 geodesic $\gamma^0$ passes within $C$ of $\go$ and has width less than $C$. Let $\gamma_-, \gamma_+\in \partial_{\rm vis} X$ be its limit points and parametrize $\gamma^0$ such that $\gamma_{0}(0)$ is at minimal distance to $\go$. By Lemma \ref{strip}  for each $i\in \mathbb{N}$ there are neighborhoods $V^{\pm}_i$ of $\gamma_\pm$ in $\partial_{\rm vis} X$ such that any pair in $V^{-}_i\times V^+_i$ can be connected by a geodesic of width less than $2C$ which passes within $2C$ of $\gamma^0(i)$. Letting 
\[
V^\pm=\bigcap^{1+\lfloor 2L\rfloor}_{i=-1-\lfloor 2L\rfloor} V_i^\pm 
\]
we have that any pair in $V^{-}\times V^+$ can be connected by a geodesic of width less than $2C$ which passes within $2C+1$ of $\gamma^0(t)$ for any $|t|\leq 2L$.
Thus, for any geodesic $\gamma$ with endpoints in $V^{-}\times V^+$ and $|t|\leq 2L$ we have $d(\gamma,\gamma^0(t))\leq 2C+1$.

Let $\Lambda'(L,C)$ be the set of all sample paths $\omega$ with $\omega^\pm\in V^\pm$.  
By definition,
\[
\Lambda'(L,C)\subset \Lambda(L,2C+1).
\]
Since the $V^\pm$ are open neighborhoods of $\gamma^{0}_{\pm}\in L(G)$ in $\partial_{\rm vis} X$ and the harmonic measure $\nu$ has 
full support on the limit set $L(G) \subset \partial_{\rm vis} X$,
we have $\nu(V^\pm)>0$, hence $\overline{P}(\Lambda'(L,C))>0$ and thus $\overline{P}(\Lambda(L,2C+1))>0$. This completes the proof 
with $2C+1$ in place of $C$.
\end{proof}
Let $C>0$ satisfy the conditions of the claim and assume wothout loss of generality that $R>L>2C$.
Note, $U^{i}\omega \in \Lambda(L,C)$ if and only if $d(\gamma_{\omega_i\go}(t),\omega_{i}\gamma^0 (t))<C$ for all $\gamma\in \Psi(\omega)$.
Let $d=\max_{g\in G,\mu(g)>0} d(g\go,\go)$. This is finite since $\mu$ is assumed to be finitely supported.
Note, we always have $d(\go,\omega_{i}\go)\leq di$ and hence if $$U^{i}\omega \in \Lambda(L,C)$$ for some $i$ with
$$0\leq i\leq \frac{R/2-L-2C}{2d}$$ then for all $\gamma\in \Psi(\omega)$, $\gamma_{\go}([-R/2,R/2])$ contains a length $L$  segment which is $C$ close to $g \cdot \gamma^0$ for some $g\in G$. We want to show that the measure of $\omega$ in the complementary set converges to $0$ as $R\to \infty$. By the Birkhoff ergodic theorem, for $\overline{P}$ almost every $\omega$, 
\[
\lim_{N\to \infty} \frac{|i\in [0,N]: U^i \omega\in \Lambda(L,C)|}{N}=\overline{P}(\Lambda(L,C))>0.
\]
In particular, for $\overline{P}$ almost every $\omega$, there is some $i$ with $U^i \omega \in \Lambda(L,C)$.
Therefore, the $\overline{P}$ measure of sample paths $\omega$ such that $U^{i}\omega \notin \Lambda(L,C)$ for all $i$ with 
$$0\leq i\leq \frac{R-L-2C}{2d}$$ converges to $0$ as $R\to \infty$ completing the proof.
\end{proof}

\section{Genericity of sublinearly Morse geodesics in Teichm\"uller space}
In this section we consider the \Teich space $\Tei(S)$ of a closed surface $S$ of genus 
$g\geq 2$ equipped with the Teichm\"uller metric $d=d_T$. Let $\MF$ be space of measured
foliations on $S$ and let $\mathcal{PMF}$ be the set of projective classes of elements of
$\mathcal {MF}$. There is a natural 
compactification of $\Tei(S)$ by $\PMF$ called the \emph{Thurston compactification} of $\Tei(S)$
The space $\MF$ can be locally parametrized
by cones in $\R^{6g-6}$. This defines a natural Lebesgue class of measures called the 
Thurston measure. This also defines a Lebesgue class of measures on $\PMF$.
(see \cite{FLP, Penner} for definitions and discussion).

Any \Teich geodesic ray can be described by a one parameter family of quadratic differentials.
The real and imaginary parts of a quadratic differential define two measures foliations
that are called the horizontal and the vertical foliation associated to the quadratic differential.  
The \Teich geodesic flow acts by scaling the horizontal foliation up and scaling the vertical 
foliation down (see \cite{Hubbard} for backgroun information on \Teich space). 

A measured foliation is called \emph{arational} if it does not contain any simple closed curve.  
It is called \emph{uniquely ergodic} if its underlying topological foliation supports a unique transverse
measure up to scaling. We also call its projective class arational or uniquely ergodic. 
Let $\mathcal{UE}\subset \mathcal{PMF}$ be the set of arational, uniquely ergodic projective measured
foliations. The $\mathcal{US}$ has full measure with respect to the Thurston measure
\cite{MasurUE}. 

A \Teich geodesic with vertical foliation in $\mathcal{UE}$ converges to its 
projective class in the Thurston compactification \cite{MasurUE} (although not always 
\cite{Lenzhen, LLR}). This allows us to consider the 
Thurston compactification as an analogue of the visual compactification of $\CAT$ spaces
(up to a measure zero set). We now proceed with the proof of Theorem \ref{fullmeasuresublinearteich}
which is really two separate statements.
The proofs are nearly identical to those in Section~\ref{Section5} and Section~\ref{Section6}. Hence, we only highlight
the differences and skip repeating the identical parts of the arguments. 

\subsection{Thurston measures on $\mathcal{PML}$}

The first measure we consider is the so-called normalized Thurston measure on $\mathcal{PMF}$, which can be considered as an analogue of the Patterson-Sullivan measure and which we now define.
There is a natural symplectic form on $MF$ which induces a measure called the Thurston measure $m^{Th}$.
For $\eta\in \mathcal{MF}$ define $[\eta]\in \mathcal{PMF}$ to be its projective class and for any basepoint $\go\in \Tei(S)$ define a \emph{normalized Thurston measure} $\nu^{Th}_{\go}$ on $\mathcal{PMF}$ by $\nu^{Th}_{\go}(A)=m(\eta\in \mathcal{MF}: [\eta]\in A, Ext_{\go}\eta\leq 1)$ where $Ext_x\eta$ denotes the extremal length of the measured foliation $\eta$ with respect to the conformal structure defimed by $x$. 

We will prove the following.
\begin{theorem}\label{Thurstonfullmeasuresublinearteich}
Let $S$ be a closed surface of genus at least $2$ and let $\Tei(S)$ be the \Teich space of $X$ with the \Teich metric. Let $\PMF$ be Thurston's boundary of \Teich space consisting of projective measured foliations.  Let $\nu$ be a measure on $\PMF$ which is a normalized Thurston measure. Then $\nu$ gives full measure to foliations 
associated to sublinearly Morse geodesics rays.
\end{theorem}
As in the proof of Theorem \ref{psarereasonable}, to prove this result we will need to prove a bilateral analogue, using an appropriate geodesic flow. 

 Any pair of distinct elements of $\mathcal{UE}$ determines a unique
Teichm\"uller geodesic with corresponding vertical and horizontal measured foliations \cite{Hubbard-Masur}.

Let $\nu$ be a normalized Thurston measure on $\mathcal{PML}$. The measure $\nu\times \nu$ gives full measure to pairs of distinct elements of $\mathcal{UE}$, and thus after taking the product with the arc-length (with respect to Teichnueller metric) normalized Lebesgue measure $L$ can be considered a measure on the space $Q^{1}$ of unit area quadratic differentials, which can be seen as the (co)tangent bundle to Teichm\"uller space. We can find a $G$-invariant and Teichm\"uller geodesic flow invariant Radon measure $\tilde{m}$ on $Q^1$ in the measure class of $\nu \times \nu \times dL$. This measure $\tilde{m}$ projects to a finite Teichm\"uller geodesic flow invariant and ergodic measure $m$ on $Q^1/\MCG(S)$, called the Masur-Veech measure. Consequently the $Mod(S)$ action on $\mathcal{PML}$ with the Thurston measure is ergodic, see \cite{ABEM} for details.
The proof now proceeds exactly as for the Patterson-Sullivan measure in the \CAT setting (see Section 5), using the following two facts:
\begin{itemize} 
\item Any two geodesic rays with the same vertical projective measured foliation in $\mathcal{UE}$ are strongly asymptotic, i.e. the distance between them converges to zero. \cite[Theorem 2]{MasurUE}. 
\item The axis in $\Tei(S)$ of any Pseudo-Anosov element is strongly contracting \cite{Minsky}.
\end{itemize}

\subsection{Stationary measures on $\mathcal{PML}$}
We now consider stationary measures for random walks on the mapping class group. 
A subgroup of $\MCG(S)$ is called non-elementary if it contains two pseudo-Anosov elements with disjoint fixed point sets in $\mathcal{PML}$.
A measure $\mu$ on $\MCG(S)$ is said to be non-elementary if the semigroup generated by its support is a non-elementary subgroup. We prove: 

\begin{theorem}\label{harmonicfullmeasuresublinearteich}
Let $S$ be a closed surface of genus at least $2$ and let $\Tei(S)$ be the \Teich space of $X$ with the \Teich metric. Let $\PMF$ be Thurston's boundary of \Teich space consisting of projective measured foliations.  Let $\nu$ be a measure on $\PMF$ which is the stationary measure associated to a finitely supported probability measure $\mu$ on the mapping class group $\MCG(S)$ such that the semigroup generated by the support of $\mu$ is a group containing at least two independent pseudo-Anosov elements. Then $\nu$ gives full measure to foliations 
associated to sublinearly Morse geodesics rays.
\end{theorem}

We now explain how to prove Theorem \ref{harmonicfullmeasuresublinearteich}.
Let $\mu$ be such a symmetric finitely supported non-elementary measure, $G<MCG(S)$ the subgroup generated by its support,  and $P$, $\overline{P}$ the induced Markov measures on unilateral and bilateral sample paths respectively. Kaimanovich-Masur \cite{Kaimanovich-Masur} proved that for $P$-a.e.  $\omega$,  and every $o\in \Tei(S)$,
$\omega_{n}\cdot \go$ converges to a uniquely ergodic point
$\omega_\infty \in \mathcal{PML}$. 
In other words, there is a $P$-a.e. where defined
measurable map ${\rm bnd}:G^{\mathbb{N}}\to {\mathcal P\mathcal M\mathcal L}$ 
sending $\omega$ to $\lim_{n\to \infty}\omega_{n} \cdot \go\in \mathcal{PML}$.
The measure on
$\mathcal{PML}$ defined by
\[\nu={\rm bnd}_{*}P=\lim_{n\to \infty}\mu^{*n}\]
is the unique $\mu$ stationary measure on $\mathcal{PML}$. In fact, 
$(\mathcal{PML},\nu)$
is a model for the \emph{Poisson boundary} of $(G,\mu)$. The measure $\nu$ gives full weight to $\mathcal{UE}$ and has full support on the limit set $L(G)\subset \mathcal{PML}$ of the group $G<MCG(S)$ generated by the support of $\mu$
\cite{Kaimanovich-Masur}.
Let $l=\lim_{n\to \infty}d(\omega_n \go,\go)/n$ (for $P$ a.e. $\omega$) be the drift of the $\mu$ random walk. Tiozzo \cite{Tiozzo} proved that $P$ a.e. $\omega$ sublinearly tracks a geodesic $\tau$ in $\Tei(S)$: $$\lim_{n\to \infty}\frac{d(\tau(l n),\omega_\go)}{n}=0$$ for any geodesic ray $\tau$ converging to $\omega_\infty \in \mathcal{PML}$. 
We want to prove that for $\nu$ a.e. $\zeta\in \mathcal{PML}$, geodesic rays with vertical projective measured foliation $\zeta$ are frequently contracting. Since any two geodesic rays with the same vertical projective measured foliation in $\mathcal{UE}$ are strongly asymptotic it suffices to prove the following analogue of Proposition \ref{longfellowtravelbilateral}.

\begin{proposition}\label{longfellowtravelbilateralteich}

Let $g_0\in G$ be a pseudo-Anosov element and $\gamma^0$ its axis in $\Tei(S)$.
Fix a basepoint $\go\in \Tei(S)$. Then there is a $K>0$ such that for $\overline{P}$ a.e. biinfinite sample path $\omega$ (any unit speed parametrization of) the biinfinite geodesic $\gamma_\omega$ satisfies the following. For any $\infty>b>a>-\infty$ and $L>0$ there is an $R_0>0$ such that for any $R>R_0$ there is a $g\in G$ such that the segment $\gamma([aR,bR])$ contains a length $L$ subsegment which is $C$ close to $g \cdot \gamma^0$.
\end{proposition}
As before, for a bilateral geodesic $\gamma$ and $p\in \Tei(S)$ we let $\gamma_p$ be a unit speed parametrization with $\gamma_\go(0)$ a point on $\gamma$ at minimal distance from $\go$. We can make this choice  in a $G$-equivariant way, i.e. so that $g\gamma_p=(g\gamma)_{gp}$.
For distinct $\zeta_1,\zeta_2\in \mathcal{UE}$ we denote by $\gamma_{\zeta_1,\zeta_2,p}$ the corresponding parametrization of $\gamma_{\zeta_1,\zeta_2}$ and for a sequence $\omega\in G^{\mathbb{Z}}$ with $\omega_{\pm n}\to \omega_\pm\in \mathcal{PML}$ we write $\gamma_\omega=\gamma_{\omega_{-},\omega_{+}}$ and $\gamma_{\omega,p}=\gamma_{\omega_{-},\omega_{+},p}$. 

Let $\Omega(L,C,R)$ be the set of sample paths $\omega\in G^{\mathbb{Z}}$ such that $\omega_\pm \in \mathcal{UE}$ are uniquely ergodic,
$d(\go,\gamma_\omega)<R/10$ and
$\gamma_{\omega,\go}[t-L,t+L]$ is contained in a $C$ neighborhood of $g \cdot \gamma^0)$, for some $g\in G$ and  $t\in (-R/2+M,R/2-M)$.
Proposition \ref{longfellowtravelbilateralteich} is deduced from the following is an analogue of Lemma \ref{decaylemma}.
\begin{lemma}\label{decaylemma2}
There is a $C>0$ such that for all $L>0$ there is an function $f$ with $\lim_{R\to \infty}f(R)=0$ and  
$\overline{P}(\Omega(L,C,R))>1-f(R)$.
\end{lemma}

The proof of Lemma \ref{decaylemma2} and the deduction from it of Proposition  \ref{longfellowtravelbilateralteich} are almost identical to (and somewhat simpler than) the corresponding proofs of Lemma \ref{decaylemma} and Proposition  \ref{longfellowtravelbilateral}. The role of the rank $1$ geodesic $\gamma^0$ will be replaced by the axis of a pseudo-Anosov mapping class, which is strongly contracting. Moreover, almost every bilateral sample path converges to distinct uniquely ergodic points of $\mathcal{PML}$ which actually determine a \emph{unique} Teichm\"uller geodesic. Using this, the arguments of Section 6 apply verbatim. The only thing worth adding, is that in place of Lemma \ref{strip} used in the proof of Lemma \ref{decaylemma} we use the following geometric result to prove Lemma \ref{decaylemma2}.

\begin{lemma}\label{Teichcontinuity} \cite[Proposition 5.1]{Klarreich}, \cite[Proposition 5.2]{Kent-Leininger}.
If a sequence $(\zeta_n,\zeta'_n)$ of pairs in $\mathcal{UE}$ converges to the distinct  pseudo-Anosov pair $(\phi_{-},\phi_{+})$, 
then the corresponding geodesics $\gamma_{\zeta_n,\zeta'_n}$ converge locally uniformly 
to the geodesic $\gamma_{\phi_{-},\phi_{+}}$ determined by $\phi_\pm$.   
\end{lemma}

\section{Fixing a sublinear function and the identification of the Poisson Boundary}

Let $X$ be either a rank-1 $\CAT$ space or a Teichm\"uller space, and $\partial X$ either the visual boundary or $\mathcal{PMF}$. Let $\nu$ be a measure on $\partial X$ satisfying the conditions of one of the Theorems \ref{psarereasonable}, \ref{harmonicarereasonable}, \ref{harmonicfullmeasuresublinearteich} or \ref{Thurstonfullmeasuresublinearteich}. 
The corresponding theorems together with Theorem~\ref{main-frequentlycontractingimpliessublinearlymorse}, 
imply that $\nu$-a.e. $\zeta \in \partial X$ is 
$\kappa$-Morse for some sublinear function $\kappa$. We show that the function $\kappa$ can in fact be chosen independent of $\zeta$:

\begin{theorem}\label{rayisMorse}
There is a single sublinear function $\kappa$ such that, for $\nu$ as in 
Theorem \ref{psarereasonable}, \ref{harmonicarereasonable}, \ref{harmonicfullmeasuresublinearteich} or \ref{Thurstonfullmeasuresublinearteich}, $\nu$-a.e. geodesic ray $\zeta$ is
$\kappa$-Morse.
\end{theorem}

\begin{proof}
Let $\Omega=\{\kappa_i\}$ be a countable collection of sublinear functions on $\mathbb{R}$ 
such that for any sublinear function $\kappa$ there is a $\kappa_i \in \Omega$ and $C>0$ with 
$\kappa \leq C\kappa_i$. Such a collection exists by the separability of the space of continued 
functions $X\to \mathbb{R}$. For each $i$, let  $A_i\subset \partial X$ be the set of $\kappa_i$-Morse rays
$\zeta \in \partial X$ (note that being $\kappa_i$-Morse is the same as being $C\kappa_i$-Morse.)

By Theorem~\ref{main-frequentlycontractingimpliessublinearlymorse} and 
the above four Theorems,  $\nu$-a.e. geodesic rays $\zeta$ is 
$\kappa_i$--Morse for some $i$. That is, $\nu(\cup A_i)=1$.  
Moreover, each $A_i$ is $G$-invariant. By (the comment after) Lemma \ref{PSandBMproperties} and Lemma \ref{stationaryproperties}, or the ergodicity of the Teichm\"uller geodesic flow, $G\curvearrowright (X,\nu)$, is ergodic. Thus $\nu(A_i)\in \{0,1\}$ for each $i$. Therefore, there is a single $A_i$ with $\nu(A_i)=1$, completing the proof.    
\end{proof}


We will now prove Corollary~\ref{Poissonboundary} and Corollary~\ref{Poissonboundaryteich}.
We first recall Theorem 6.2 in  \cite{QRT2}.

\begin{theorem} \label{T:poiss-general}
Let $G$ be a countable group of isometries of a proper, geodesic, metric space $(X, d_X)$, and suppose that the action 
of $G$ on $X$ is temperate.
Let $\mu$ be a probability measure on $G$ with finite first moment with respect to $d_X$, such that the semigroup generated 
by the support of $\mu$ is a non-amenable group. 
Let $\kappa$ be a concave sublinear function, and suppose that
for almost every sample path $\omega = (w_n)$, there exists a $\kappa$-Morse geodesic ray $\gamma_\omega$ such that 
\begin{equation} \label{E:sub-track}
\lim_{n \to \infty} \frac{d_X(w_n \cdot \go, \gamma_\omega)}{n} = 0.
\end{equation}
Then almost every sample path converges to a point in $\pka X$, and moreover the space $(\pka X, \nu)$, where $\nu$ is the hitting measure for the random walk, is a model for the Poisson boundary of $(G, \mu)$. 
\end{theorem}

\begin{proof}[Proof of Corollary~\ref{Poissonboundary} and Corollary~\ref{Poissonboundaryteich}]

It is known that when $X$ is \CAT or Teichm\"uller space and the semigroup $G$ generated by the support of $\mu$ is non-elementary, $\mu$-a.e. sample path sublinearly tracks some $X$ geodesic ray $\tau_\omega$ (\cite{Tiozzo} for Teichm\"uller space and \cite{KarlssonMargulis} for \CAT spaces). Moreover, the mapping class group action on Teichm\"uller space is temperate by the main result of \cite{ABEM}. Theorem~\ref{rayisMorse} shows that there is a sublinear function $\kappa$ such that for $\mu$-a.e. sample path $\omega$, $\omega_n \go$ sublinearly tracks a $\kappa$-Morse geodesic ray $\zeta$. Thus the assumptions of Theorem~\ref{T:poiss-general} are satisfied and hence $(\pka X, \nu)$ is a model for the associated Poisson boundary.
\end{proof}

%
%
%
%
%

 \bibliography{bib}
\bibliographystyle{alpha}

\end{document}